\documentclass[11pt, leqno]{amsart}
\usepackage{a4wide}
\usepackage[mathscr]{eucal}
\usepackage{amsmath}
\usepackage{amssymb}
\usepackage{amsfonts}
\usepackage{amsthm}
\usepackage{color}
\theoremstyle{plain}
\newtheorem{theorem}{Theorem}[section]
\newtheorem{proposition}{Proposition}[section]%[theorem]
\newtheorem{corollary}{Corollary}[section]%[theorem]
\newtheorem{lemma}{Lemma}[section]%[theorem]
\newtheorem{remark}{Remark}

\newtheorem{example}{Example}[section]
\newtheorem{definition}{Definition}[section]

\renewcommand{\labelenumi}{(\theenumi)}
\numberwithin{equation}{section}

\title[Homogeneous Riemannian structures]{Homogeneous structures of $3$-dimensional Lie groups}

\author[J.~Inoguchi]{Jun-ichi Inoguchi}
\address{Department of Mathematics, 
Hokkaido University, 
Sapporo 
060-0810, Japan}
\email{inoguchi@math.sci.hokudai.ac.jp}
\thanks{The first named author is partially supported by JSPS KAKENHI JP23K03081.}

\author[Y.~Ohno]{Yu Ohno}
\address{Department of Mathematics, Hokkaido University, 
Sapporo, 060-0810, Japan}
\email{ono.yu.h4@elms.hokudai.ac.jp}
\thanks{The second named author is supported by JST SPRING, Grant Number JPMJSP2119.}

\begin{document}

\begin{abstract}
We give a classification of homogeneous Riemannian structures on (non locally symmetric) 
$3$-dimensional Lie groups equipped with left invariant Riemannian metrics. 
This work together with classifications due to 
previous works yields a complete classification of all the homogeneous
Riemannian structures on homogeneous Riemannian $3$-spaces. 
Two applications of the classification to contact Riemannian geometry and CR geometry are also 
given.
\end{abstract}


\maketitle

\section{Introduction}\label{sec:1}
Riemannian symmetric spaces 
are characterized as reductive homogeneous Riemannian 
spaces whose canonical connections coincide with 
Levi-Civita connections. 
On a non-symmetric 
reductive homogeneous Riemannian 
space $(M,g)$, the canonical connection 
$\tilde{\nabla}$ satisfies 
the following properties:
\begin{equation}\label{eq:0.1}
\tilde{\nabla}g=0,
\quad \tilde{\nabla}R=0,
\quad \tilde{\nabla}S=0,
\end{equation}
where $S=\tilde{\nabla}-\nabla$ 
is the difference tensor field 
of the canonical connection 
$\tilde{\nabla}$ and the Levi-Civita 
connection $\nabla$. 
Ambrose and Singer \cite{AS} proved that
local homogeneity of a Riemannian manifold 
$(M,g)$ is characterized by the existence 
of a tensor field $S$ of type $(1,2)$ 
satisfying \eqref{eq:0.1} with respect to the 
connection $\tilde{\nabla}=\nabla+S$ (called the 
\emph{Ambrose-Singer connection}). The tensor field 
$S$ is called a \emph{homogeneous Riemannian structure}.
Tricerri and Vanhecke \cite{TV} developed a general theory 
of homogeneous Riemannian structures. After the publication of 
\cite{TV}, homogeneous Riemannian structures have been studied 
extensively. For further information we refer to \cite{CalLo}.  
Podest{\'a} \cite{Podesta} generalized Ambrose-Singer's characterization 
of locally homogeneous Riemannian spaces 
to cohomogeneity one Riemannian manifolds. See also 
a recent work \cite{CCD}.

As is well known, 
complete homogeneous Riemannian $3$-manifolds 
are either Riemannian symmetric or 
\emph{Riemannian groups}, i.e., 
Lie groups equipped with left invariant metrics. 

The $3$-dimensional simply connected Riemannian 
symmetric spaces are 
\begin{itemize}
\item irreducible spaces: $\mathbb{S}^3(c^2)$, 
$\mathbb{E}^3$, $\mathbb{H}^3(-c^2), $ 
\item reducible spaces: $\mathbb{S}^2(c^2)\times\mathbb{E}^1$, 
$\mathbb{H}^2(-c^2)\times\mathbb{E}^1$.
\end{itemize}
The homogeneous Riemannian structures of space forms ($\mathbb{S}^3(c^2)$, 
$\mathbb{E}^3$, $\mathbb{H}^3(-c^2)$) were classified by Abe \cite{Abe}. 
The second named author of the present paper classified 
homogeneous Riemannian structures of $\mathbb{S}^2(c^2)\times\mathbb{E}^1$
and $\mathbb{H}^2(-c^2)\times\mathbb{E}^1$ \cite{Ohno}.

Recently, Calvi{\~n}o-Louzao, 
Ferreiro-Subrido, 
Garc{\'\i}a-R{\'\i}o and 
V{\'a}zquez-Lorenzo \cite{CFG} studied 
homogenous Riemannian structures on 
non-locally symmetric $3$-dimensional Riemannian groups. 
It was not explicitly stated, but 
they classified homogeneous Riemannian structures on 
those Lie groups under the assumption 
that homogeneous Riemannian structures 
are left invariant with respect to the Lie group structure. 

The purpose of this article is to improve the classification 
due to \cite{CFG} for $3$-dimensional unimodular Lie groups as well as
$3$-dimensional non-unimodular Lie groups. More precisely 
we remove the left invariance assumptions 
of the classification due to \cite{CFG}. 

Let us detail our results on unimodular Lie groups (1) and 
non-unimodular Lie groups (2).

\medskip

\noindent (1) Unimodular Lie groups:

Now let $G$ be a $3$-dimensional unimodular Lie group 
with unimodular basis $\{e_1,e_2,e_3\}$ of the Lie algebra $\mathfrak{g}$ of $G$ satisfying 
\[
[e_1,e_2]=c_3e_3,\quad 
[e_2,e_3]=c_1e_1,\quad 
[e_3,e_1]=c_2e_2.
\]
The possible dimension of the isometry group $\mathrm{Iso}(G)$ of $G$ is 
$3$, $4$ or $6$. 
It is known that $\mathrm{Iso}(G)$ is $4$-dimensional 
if and only if $c_1=c_2\not=c_3$ and $c_3\neq0$ up to numeration. 
Moreover $\mathrm{Iso}(G)$ is $6$-dimensional if and 
only if $c_1=c_2=c_3$ or $c_1=c_2\not=c_3=0$ up to numeration. 
Note that the Lie group $G$ is 
isometric to $\mathbb{S}^3$ or $\mathbb{R}P^3$ if $c_1=c_2=c_3>0$ and 
$G$ is 
locally isometric to $\mathbb{R}^3$ if $c_1=c_2=c_3=0$ or $c_1=c_2\not=c_3=0$,
and then in this case the homogeneous Riemannian structures of $G$
is classified by Abe \cite{Abe}.

In case $c_1=c_2\not=c_3$ and $c_3\not=0$, all the 
\emph{left invariant} homogeneous Riemannian structures 
of $G$ are given by \cite[Theorem 1.2]{CFG}
\begin{equation}\label{eq:S4}
S_{\flat}=-c_3\theta^1\otimes(\theta^2\wedge\theta^3)
-c_3\theta^2\otimes(\theta^3\wedge\theta^1)
-2r\theta^3\otimes(\theta^1\wedge\theta^2),
\quad r\in\mathbb{R},
\end{equation}
where $\{\theta^1,\theta^2,\theta^3\}$ is the left 
invariant coframe field metrically 
dual to $\{e_1,e_2,e_3\}$.
Here $S_{\flat}$ is the covariant form of $S$ 
that is defined by
\[
S_{\flat}(X,Y,Z)=g(S(X)Y,Z).
\]
In our previous work \cite{IO} we improved and generalized 
Theorem 1.2-(ii) of \cite{CFG}. 
More precisely, we proved the following theorem.

\begin{theorem}\label{thm:uni_hs}
Let $G$ be a $3$-dimensional unimodular Lie 
group admitting a unimodular basis of the Lie algebra $\mathfrak{g}$ satisfying 
$c_1=c_2\not=c_3$ and $c_3\not=0$. Then
\begin{enumerate}
\item if the Lie algebra $\mathfrak{g}$ is 
not isomorphic to $\mathfrak{sl}_2\mathbb{R}$, then any 
homogeneous Riemannian structure is left invariant. 
All of those are given by \eqref{eq:S4}.
\item if the Lie algebra $\mathfrak{g}$ is 
isomorphic to $\mathfrak{sl}_2\mathbb{R}$, 
then there exists a unique non-left invariant 
homogeneous structure $S$. 
The corresponding 
coset space representation is $G=L^{\prime}/\{\mathsf{e}\}$. Here 
$L^{\prime}$ is a $3$-dimensional non-unimodular Lie group whose Lie algebra is isomorphic to 
the Lie algebra $\mathfrak{ga}(1)\oplus\mathbb{R}$, where 
$\mathfrak{ga}(1)$ is the Lie algebra of the orientation 
preserving affine transformation group 
$\mathrm{GA}^{+}(1)=\mathrm{GL}^{+}_{1}\mathbb{R}\ltimes\mathbb{R}$ of the 
real line. The homogeneous Riemannian structure $S$ is expressed as 
$S=\nabla^{(-)}-\nabla$,
where $\nabla^{(-)}$ is a Cartan-Schouten's $(-)$-connection for 
$L^{\prime}$.

\medskip

All of homogeneous 
Riemannian structures on $G$ are given by \eqref{eq:S4} and 
$S$.
\end{enumerate}
\end{theorem}
As a continuation of \cite{IO}, in this article, we 
study homogeneous Riemannian structures 
on $3$-dimensional unimodular Lie groups.
Our theorem below improves Theorem 1.2-(i) of 
\cite{CFG}.
\begin{theorem}
Let $G$ be a $3$-dimensional unimodular Lie 
group admitting a unimodular basis satisfying 
all distinct structure constants 
$c_1$, $c_2$ and $c_3$. 
Then there exists a unique 
homogeneous Riemannian structure 
$S$ on $G$. The homogeneous Riemannian structure 
$S$ is left invariant and its Ambrose-Singer connection 
coincides with Cartan-Schouten's $(-)$-connection.
\end{theorem}

\medskip 

\noindent 
(2) Non-unimodular Lie groups:

Let $G$ be a $3$-dimensional non-unimodular Lie group. Then 
we may assume that there exists an orthonormal basis $\{e_1,e_2,e_3\}$ of the Lie algebra 
of $G$ satisfying 
the commutation relations:
\[
[e_1,e_2]=(1+\alpha)e_2+(1+\alpha)\beta e_3,
\quad [e_2,e_3]=0,
\quad 
[e_3,e_1]=(1-\alpha)\beta e_2-(1-\alpha)e_3,
\]
where $\alpha$ and $\beta$ are non-negative constants. 
The left invariant metric is locally symmetric if and only if $\alpha=0$ or 
$(\alpha,\beta)=(1,0)$. 
To emphasize the fact that the Lie algebra of $G$ is 
determined by structure constants $\alpha$ and $\beta$, we denote 
by the Lie algebra of $G$ by $\mathfrak{g}(\alpha,\beta)$. 
Moreover we express $G$ as 
$G=G(\alpha,\beta)$.

\begin{theorem}\label{thm:3}
Let $G=G(\alpha,\beta)$ be a $3$-dimensional non-unimodular Lie 
group which is not locally symmetric. Then
\begin{enumerate}
\item if $\alpha\neq 1,$ then the only homogeneous Riemannian structure is $S=\nabla^{(-)}-\nabla$, 
where $\nabla^{(-)}$ is Cartan-Schouten's $(-)$-connection. The coset representation is $G=G/\{\mathsf{e}\}$.
\item if $\alpha=1,$ then the homogeneous Riemannian structure of $G=G(1,\beta)$ is $S=\nabla^{(-)}-\nabla$ or
\[
S^{(r)}=-2r\theta^3\otimes(\theta^1\wedge\theta^2)
-2\beta\theta^1\otimes(\theta^2\wedge\theta^3)
-2\beta\theta^2\otimes(\theta^3\wedge\theta^1),
\quad r\in\mathbb{R}.
\]
The coset space representations of $G$ corresponding to $S$ 
is $G=G/\{e\}$ and the coset space representation 
corresponding to $S^{(r)}$ are 
\[
G=\begin{cases}
L\times\mathrm{SO}(2)/\mathrm{SO}(2), & r\not=-(\beta^2+2)/\beta,
\\
L/\{\mathsf{e}\}, & r=-(\beta^2+2)/\beta,
\end{cases}
\]
where $L$ is a Lie group acting transitively and isometrically on $G$ and 
its Lie algebra $\mathfrak{l}$ is isomorphic to $\mathfrak{sl}_2\mathbb{R}$.
\end{enumerate}
\end{theorem}
Theorem \ref{thm:3} improves Theorem 1.3 of 
\cite{CFG}. In Theorem \ref{thm:3}-2, the non-unimodular Lie group 
$G(1,\beta)$ with $\beta>0$ is locally isometric to 
$\mathrm{SL}_2\mathbb{R}$ equipped with a suitable left invariant metric 
as in Theorem \ref{thm:uni_hs}-2. Although the homogeneous Riemannian structure $S$ on 
$G(1,\beta)$ given in Theorem \ref{thm:3}-2 is not left invariant with respect to 
the Lie group structure of $G(1,\beta)$, but left invariant with respect to the 
Lie group structure of $L$ with Lie algebra $\mathfrak{l}\cong\mathfrak{sl}_2\mathbb{R}$.

Conversely, the non-left invariant homogeneous Riemannian structure 
on a uni-modular Lie group $G$ described in Theorem \ref{thm:uni_hs}-2 is 
left invariant with respect to the 
Lie group structure of $L^{\prime}$. The Lie algebra 
$\mathfrak{l}^{\prime}$ is isomorphic to $\mathfrak{ga}(1)\oplus\mathbb{R}$.  
The Lie algebra $\mathfrak{ga}(1)\oplus\mathbb{R}$ is isomorphic to the non-unimodular Lie algebra 
$\mathfrak{g}(1,\beta)$ for some $\beta>0$ (see \cite{IO} or Example \ref{eg:4.3}). 
Thus Theorem \ref{thm:uni_hs}-2 and Theorem \ref{thm:3}-2 are 
in twin.

In particular the universal covering $\widetilde{G}(1,\beta)$ of $G(1,\beta)$ with $\beta>0$ 
is isometric to the universal covering $\widetilde{\mathrm{SL}}_2\mathbb{R}$ of 
$\mathrm{SL}_2\mathbb{R}$. As a Riemannian manifold, 
$\widetilde{G}(1,\beta)$ is identical to $\mathrm{SL}_2\mathbb{R}$.  
However this Riemannian $3$-manifold admits non-isomorphic 
homogeneous space representations 
\[
(\widetilde{\mathrm{SL}}_2\mathbb{R}\times\mathrm{SO}(2))/\mathrm{SO}(2),
\quad 
\widetilde{\mathrm{SL}}_2\mathbb{R}/\{\mathsf{e}\},
\quad 
\widetilde{G}(1,\beta)/\{\mathsf{e}\}.
\]  
Thus Theorem \ref{thm:uni_hs}-2 and Theorem \ref{thm:3}-2 are 
unified in this way.

This work together with classifications due to 
previous works \cite{Abe,CFG,GO,IO,Ohno,TV} yields a complete classification of all the homogeneous
Riemannian structures on homogeneous Riemannian $3$-spaces. 

This paper ends with two applications of our results to contact Riemannian geometry and CR geometry.

\section{Homogeneous Riemannian spaces}\label{sec:2}
\subsection{Ambrose and Singer connections}
Let $(M,g)$ be a Riemannian manifold with Levi-Civita connection
$\nabla$. The \emph{Riemannian curvature} $R$ is defined by
\[
R(X,Y)=[\nabla_{X},\nabla_{Y}]-\nabla_{[X,Y]}.
\]
The Ricci tensor field $\mathrm{Ric}$ is defined by
\[
\mathrm{Ric}(X,Y)=\mathrm{tr}_{g}(Z\longmapsto R(Z,Y)X).
\]
\begin{definition}
{\rm
A Riemannian manifold $(M,g)$ is said to be a 
\emph{homogeneous Riemannian space} if there exists a Lie group $G$ of 
isometries which acts transitively on $M$.

More generally, $M$ is said to be \emph{locally homogeneous Riemannian 
space} if for each $p$, $q\in M$, 
there exists a local isometry which sends $p$ to 
$q$.
}
\end{definition}

Ambrose and Singer \cite{AS} 
gave an \emph{infinitesimal characterization} of 
local homogeneity of Riemannian manifolds. 
To explain their characterization we recall the following notion:

\begin{definition}{\rm
A \emph{homogeneous Riemannian structure} $S$ on $(M,g)$ is
a tensor field of type $(1,2)$ which satisfies
\begin{equation}
\tilde{\nabla}{g}=0,
\quad 
\tilde{\nabla}{R}=0,
\quad 
\tilde{\nabla}{S}=0.
\end{equation}
Here $\tilde{\nabla}$ is a linear connection on $M$ defined
by $\tilde{\nabla}=\nabla+S$. The linear connection $\tilde{\nabla}$ is called 
the \emph{Ambrose-Singer connection}.
}
\end{definition}
\begin{remark}{\rm
The sign conventions of $R$ and $S$ are opposite 
to the ones used in \cite{TV,CalLo}. 
}
\end{remark}

Let $(M,g)=G/H$ be a homogeneous Riemannian space.
Here $G$ is a connected Lie group acting transitively
on $M$ as a group of isometries. 
Without loss of generality we can assume that $G$ acts 
\emph{effectively} on $M$. 
The subgroup $H$ is the isotropy subgroup
of $G$ at a point $o\in M$ which will be called the \emph{origin} of $M$. 
Moreover we may assume that $G/H$ is \emph{reductive}, 
that is, there exists a linear subspace $\mathfrak{m}$ of 
the Lie algebra $\mathfrak{g}$ of $G$ satisfying
\[
\mathfrak{g}=\mathfrak{h}\oplus\mathfrak{m},\quad
[\mathfrak{h},\mathfrak{m}]\subset\mathfrak{m},
\]
where $\mathfrak{h}$ is the Lie algebra of $H$. 
Indeed, the following result is 
known (see \textit{e.g.}, \cite{KoSz}):
\begin{proposition}
Any homogenous Riemannian space $(M,g)$ is a reductive homogenous space.
\end{proposition}
Note that this proposition does not hold 
for homogeneous semi-Riemannian spaces equipped 
with \emph{indefinite metric} (see \textit{eg.} \cite[Chapter 7]{CalLo}). 

If $H$ is connected, the reductive 
condition $[\mathfrak{h},\mathfrak{m}]\subset
\mathfrak{m}$ is equivalent to the 
$\mathrm{Ad}(H)$-invariance of
$\mathfrak{m}$. For any $a\in G$, 
the translation $\tau_{a}$ by $a$ is 
a diffeomorphism on $M$ defined by $\tau_{a}(bH)=(ab)H$. 
The tangent space 
$\mathrm{T}_{p}M$ of $M$ at a point 
$p=\tau_a(o)$ is identified with $\mathfrak{m}$ via the isomorphism
\[
\mathfrak{m}\ni
X
\longleftrightarrow 
X^{*}_{p}=\frac{\mathrm{d}}{\mathrm{d}t}\biggr \vert_{t=0}
\tau_{\exp(tX)}(p)
\in\mathrm{T}_{p}M.
\]
The \emph{canonical connection} 
$\tilde{\nabla}=\nabla^{\mathrm c}$ of $G/H$ with respect to the 
reductive decomposition $\mathfrak{g}=\mathfrak{h}\oplus\mathfrak{m}$ is given by
\[
(\tilde{\nabla}_{X^*}Y^{*})_o
=-([X,Y]_{\mathfrak m})^{*}_o,
\quad 
X,Y 
\in \mathfrak{m}.
\]
For any vector $X\in\mathfrak{g}$, 
the $\mathfrak{m}$-component of $X$ is 
denoted by $X_{\mathfrak{m}}$.  
One can see the difference tensor field $S=\tilde{\nabla}-\nabla$
is a homogeneous Riemannian structure.
Thus every homogeneous Riemannian space admits homogeneous Riemannian structures.

Conversely, let $(M,S)$ be a \emph{simply connected} Riemannian manifold with a
homogeneous Riemannian structure $S$. Fix a point $o\in M$ and put 
$\mathfrak{m}=\mathrm{T}_{o}M$. Denote by $\tilde{R}$ the curvature of the 
Ambrose-Singer connection $\tilde{\nabla}$. Then the holonomy algebra
$\mathfrak{h}$ of $\tilde{\nabla}$ 
is the Lie subalgebra of the Lie algebra 
$\mathfrak{so}(\mathfrak{m},g_o)$ 
generated by the curvature 
operators $\tilde{R}(X,Y)$ with 
$X$, $Y\in\mathfrak{m}$.

Now we define a Lie algebra structure on the direct sum 
$\mathfrak{g}=\mathfrak{h}\oplus \mathfrak{m}$
by (see \cite{AS,CalLo,TV}):
\begin{align}
[U,V]=& UV-VU,
\notag \\
[U,X]=& U(X),
\label{eq:hm-bla}
\\
[X,Y]=& -\tilde{R}(X,Y)-S(X)Y+S(Y)X
\label{eq:m-bla}
\end{align}
for all $X$, $Y\in \mathfrak{m}$ and $U$, $V\in \mathfrak{h}$.

Now let $\tilde{G}$ be the simply connected Lie group with
Lie algebra $\mathfrak{g}$. Then $M$ is a coset manifold $\tilde{G}/\tilde{H}$, where
$\tilde{H}$ is a Lie subgroup of $\tilde{G}$ with Lie algebra $\mathfrak{h}$.
Let $\Gamma$ be the set of all elements in $G$ which act trivially on $M$. Then
$\Gamma$ is a discrete normal subgroup of $\tilde{G}$ and $G=\tilde{G}/\Gamma$
acts transitively and effectively on $M$ as an isometry group.
The isotropy subgroup $H$ of $G$ at $o$ is $H=\tilde{H}/\Gamma$.  
Hence $(M,g)$ is a homogeneous Riemannian space with coset
space representation $M=G/H$.

\begin{theorem}[\cite{AS}]
A Riemannian manifold $(M,g)$ with a homogeneous Riemannian structure $S$
is locally homogeneous. 
\end{theorem}
 
\begin{definition}
{\rm Let $(M,g,S)$ and $(M^{\prime},g^{\prime},S^{\prime})$ be 
Riemannian manifolds equipped with homogeneous Riemannian structures. 
Then $(M,g,S)$ and $(M^{\prime},g^{\prime},S^{\prime})$ are said to be 
\emph{isomorphic each other} if there exists an isometry 
$\psi:M\to M^{\prime}$ satisfying $\psi^{*}S^{\prime}=S$.
}
\end{definition}

\begin{theorem}[\cite{TV}]\label{thm:hom}
Let $(M,g,S)$ be a homogeneous Riemannian space and $G,G'$ connected Lie subgroups of 
the identity component $\mathrm{Iso}_{\circ}(M,g)$ of the full isometry group acting transitively
on $M$. Assume that the Lie algebras $\mathfrak{g}$ of $G$ and $\mathfrak{g}'$ of $G'$
has reductive decompositions $\mathfrak{g}=\mathfrak{h}\oplus\mathfrak{m}$ and $\mathfrak{g'}=\mathfrak{h'}\oplus\mathfrak{m'},$
respectively. Then the homogeneous Riemannian structures $S$ determined by $\mathfrak{g}=\mathfrak{h}\oplus\mathfrak{m}$
and $S'$ determined by $\mathfrak{g'}=\mathfrak{h'}\oplus\mathfrak{m'}$ are isomorphic if and only if
there exists a Lie algebra isomorphism $F\colon\mathfrak{g}\to\mathfrak{g}'$ such that
\[
F(\mathfrak{h})=\mathfrak{h}', \quad F(\mathfrak{m})=\mathfrak{m}'
\]
and $F|_{\mathfrak{m}}$ is a linear isometry.
\end{theorem}

\subsection{The eight classes of homogeneous Riemannian structures}
For a homogeneous Riemannian structure $S$ on a Riemannian manifold 
$(M,g)$, we denote by $S_{\flat}$ the covariant tensor field metrically 
equivalent to $S$, that is,
\[
S_{\flat}(X,Y,Z)=g(S(X)Y,Z)
\]
for all $X$, $Y$, $Z\in \varGamma(\mathrm{T}M)$.
The metrical condition 
$\tilde{\nabla}g=0$ is rewritten as
\begin{equation}\label{eq:AS-met}
S_{\flat}(X,Y,Z)+S_{\flat}(X,Z,Y)=0
\end{equation}
for all vector fields $X$, $Y$ and $Z$.

Tricerri and Vanhecke \cite{TV} obtained the following decompositions 
of all possible types of homogeneous 
Riemannian structures into eight classes 
(Table \ref{table1}):

\medskip
\begin{table}[h]
%\begin{center}
\centering
\begin{tabular}{|l|l|}
\hline
Classes & Defining conditions  \\
\hline
Symmetric   & $S=0$    \\
\hline
$\mathcal{T}_1$   & $S_{\flat}(X,Y,Z)=g(X,Y)\omega(Z)-g(Z,X)\omega(Y)$  for some $1$-form $\omega$\\
\hline
$\mathcal{T}_2$ & $\underset{X,Y,Z}{\mathfrak{S}}\, S_{\flat}(X,Y,Z)=0$ and $c_{12}(S_\flat)=0$ \\
\hline
$\mathcal{T}_3$ & $S_{\flat}(X,Y,Z)+S_{\flat}(Y,X,Z)=0$  \\
\hline
$\mathcal{T}_{1}\oplus \mathcal{T}_2$ & $\underset{X,Y,Z}{\mathfrak{S}}\, S_{\flat}(X,Y,Z)=0$ \\
\hline
$\mathcal{T}_{1}\oplus \mathcal{T}_3$
&
$
S_{\flat}(X,Y,Z)+S_{\flat}(Y,X,Z)=2g(X,Y)\omega(Z)-g(Z,X)\omega(Y)-g(Y,Z)\omega(X)
$
\\
{} &
for some $1$-form $\omega$\\
\hline  
$\mathcal{T}_2\oplus \mathcal{T}_3$ & $c_{12}(S_\flat)=0$ \\
\hline
$\mathcal{T}_{1}\oplus\mathcal{T}_2\oplus \mathcal{T}_3$
& no conditions
\\
\hline
\end{tabular}
\smallskip
\caption{The eight classes}
\label{table1}
\end{table}

Here $\underset{X,Y,Z}{\mathfrak{S}}\,S_{\flat}$ denotes the cyclic sum of $S_\flat$, \textit{i.e.},
\[
\underset{X,Y,Z}{\mathfrak{S}}\, S_{\flat}(X,Y,Z)=
S_{\flat}(X,Y,Z)+S_{\flat}(Y,Z,X)+S_{\flat}(Z,X,Y).
\]
Next $c_{12}$ denotes the contraction operator in $(1,2)$-entries;
\[
c_{12}(S_{\flat})(Z)=\sum_{i=1}^{n}S_{\flat}(e_i,e_i,Z),
\] 
where $\{e_1,e_2,\dots,e_n\}$ is a local orthonormal frame field of $(M,g)$.
Tricceri and Vanhecke \cite{TV} classified 
homogeneous Riemannian $3$-manifolds 
admitting homogeneous Riemannian structure of type $\mathcal{T}_3$. 
On the other hand, Kowalski and Tricceri \cite{KoTr} 
classified 
homogeneous Riemannian $3$-manifolds 
admitting homogeneous Riemannian structure of type $\mathcal{T}_2$.

\begin{proposition}[\cite{TV,Pastore}]
Let $(M,g)$ be a Riemannian manifold 
admitting a non trivial 
homogeneous Riemannian structure $S$. Then 
\begin{enumerate}
\item $S$ is of type $\mathcal{T}_1$ if and 
only if 
\[
S(X)Y=-g(X,Y)\xi+g(\xi,Y)X,\quad 
\xi=-\frac{1}{n-1}\mathrm{tr}_{g}S.
\]
Moreover if $M$ admits a 
non trivial homogeneous Riemannian 
struture of type $\mathcal{T}_1$ then 
$M$ is 
of constant curvature $-g(\xi,\xi)$.
\item $S$ is of type $\mathcal{T}_1\oplus\mathcal{T}_3$ if and 
only if 
\[
S(X)Y+S(Y)X=-2g(X,Y)\xi+g(\xi,Y)X+g(\xi,X)Y,\quad 
\xi=-\frac{1}{n-1}\mathrm{tr}_{g}S.
\]
In particular, if 
a non trivial 
homogeneous Riemannian structure $S$ is of 
type $\mathcal{T}_1
\oplus\mathcal{T}_3$ but not 
of type $\mathcal{T}_3$, then
the $1$-form $\omega$ metrically dual 
to $\xi$ is closed. Moreover $M$ is 
of constant curvature $-g(\xi,\xi)$.
\end{enumerate}
\end{proposition}
\begin{remark}
{\rm 
In the terminology of \cite{GGO1,GGO2}, 
a non-trivial 
homogeneous Riemannian structure $S$ 
is said to be 
\begin{itemize}
\item \emph{vectorial} if it is of type $\mathcal{T}_1$.
\item \emph{cyclic} if 
it is of type $\mathcal{T}_1\oplus\mathcal{T}_2$. 
\item \emph{traceless} if 
it is of type $\mathcal{T}_2\oplus\mathcal{T}_3$.
\item \emph{traceless cyclic} if 
it is of type $\mathcal{T}_2$.
\end{itemize}
On the other hand, a homogeneous Riemannian structure of type 
$\mathcal{T}_1$ is called a 
\emph{homogeneous structure of linear type} 
in \cite{CalLo}.
}
\end{remark}
\subsection{Canonical connections on Lie groups}\label{sec:2.4}
We can regard a Lie group $G$ as a homogeneous space in two ways: 
$G=G/\{\mathsf{e}\}$ and $G=(G\times G)/\Delta G$. 
Here $\mathsf{e}$ is the unit element of $G$. 
In the first representation, the isotropy algebra is 
$\{0\}$ and the tangent space $\mathfrak{m}$ at $\mathsf{e}$ is identified with
$\mathfrak{g}$. 
Obviously the splitting $\mathfrak{g}=\{0\}+
\mathfrak{g}$ is reductive. The natural projection 
$\pi:G\to G/\{\mathsf e\}$ is the identity map.
The canonical connection of $G/\{\mathsf{e}\}$ is 
denoted by $\nabla^{(-)}$ and given by 
\[
\nabla^{(-)}_{X}Y=0,\quad X,Y\in\mathfrak{g}.
\]
The torsion $T^{(-)}$ of $\nabla^{(-)}$ is given by 
$T^{(-)}(X,Y)=-[X,Y]$. The canonical connection 
$\nabla^{(-)}$ is also called the \emph{Cartan-Schouten's $(-)$-connection} 
\cite{KN2}.

Next, let us take the product Lie group $G\times G$. 
The Lie algebra of $G\times G$ is 
\[
\mathfrak{g}=\{(X,Y)\>|\>X,Y\in\mathfrak{g}\}
\]
with Lie bracket
\[
[(X_1,Y_1),(X_2,Y_2)]=([X_1,Y_1],[X_2,Y_2]).
\]
The product Lie group $G\times G$ 
acts on $G$ by the action:
\begin{equation}\label{eq:biinvariantaction}
(G\times G)\times G\to G;\quad (a,b)x=axb^{-1}.
\end{equation}
The isotropy subgroup at the identity $\mathsf{e}$ is the 
\emph{diagonal subgroup} 
\[
\Delta G=\{(a,a)\>|\>a\in G\} 
\]
with Lie algebra $\Delta\mathfrak{g}=\{(X,X)\>|\>X\in\mathfrak{g}\}$. 
We can consider the following three 
Lie subspaces;
\[
\mathfrak{m}^{+}=\{(0,X)\>|\>X\in\mathfrak{g}\},
\quad 
\mathfrak{m}^{-}=\{(X,0)\>|\>X\in\mathfrak{g}\},
\quad 
\mathfrak{m}^{0}=\{(X,-X)\>|\>X\in\mathfrak{g}\}.
\]
Then $\mathfrak{g}\oplus\mathfrak{g}=\Delta\mathfrak{g}\oplus \mathfrak{m}^{+}$,
$\mathfrak{g}\oplus\mathfrak{g}=\Delta\mathfrak{g}\oplus \mathfrak{m}^{-}$ 
and 
$\mathfrak{g}\oplus\mathfrak{g}=\Delta\mathfrak{g}\oplus \mathfrak{m}^{0}$ are reductive.

The corresponding splittings are given 
explicitly by
\begin{align*}
&(X,Y)=(X,X)+(0,-X+Y)\in \Delta\mathfrak{g}\oplus \mathfrak{m}^{+}
\\
&(X,Y)=(Y,Y)+(X-Y,0)\in \Delta\mathfrak{g}\oplus \mathfrak{m}^{-}
\\
&(X,Y)=\left(\frac{X+Y}{2},\frac{X+Y}{2}\right)+
\left(\frac{X-Y}{2},-\frac{X-Y}{2}\right)\in \Delta\mathfrak{g}\oplus \mathfrak{m}^{0}.
\end{align*}
Let us identify the tangent space 
$\mathrm{T}_{\mathsf{e}}G$ of $G$ at $\mathsf{e}$ with these 
Lie subspaces. Then the canonical connection with respect to the 
reductive decomposition 
$\mathfrak{g}\oplus\mathfrak{g}=\Delta\mathfrak{g}\oplus \mathfrak{m}^{+}$
is denoted by $\nabla^{(+)}$ and given by
\[
\nabla^{(+)}_{X}Y=[X,Y],\quad X,Y\in\mathfrak{g}.
\]
The torsion $T^{(+)}$ of $\nabla^{(+)}$ is given by 
$T^{(+)}(X,Y)=[X,Y]$. 
The canonical connection 
$\nabla^{(+)}$ is called the \emph{Cartan-Schouten's $(+)$-connection} 
or \emph{anti canonical connection} \cite{DITAM}.

Next, the 
canonical connection with respect to the 
reductive decomposition 
$\mathfrak{g}\oplus\mathfrak{g}=\Delta\mathfrak{g}\oplus \mathfrak{m}^{-}$ is 
$\nabla^{(-)}$. Finally  
the canonical connection with respect to the 
reductive decomposition 
$\mathfrak{g}\oplus\mathfrak{g}=\Delta\mathfrak{g}\oplus \mathfrak{m}^{0}$
is denoted by $\nabla^{(0)}$ and given by
\[
\nabla^{(0)}_{X}Y=\frac{1}{2}[X,Y],\quad X,Y\in\mathfrak{g}.
\]
The connection $\nabla^{(0)}$ is torsion free and called 
the \emph{Cartan-Schouten's $(0)$-connection}, 
the \emph{natural torsion free connection} \cite{Nomizu} or  
\emph{neutral connection} \cite{DITAM}.

The following result is well known (see \textit{e.g.}, \cite{KN2,TV}).
\begin{proposition}
Let $G$ be a Lie group equipped with a 
left invariant metric $g$ with Levi-Civita 
connection $\nabla$. Then the following 
holds.
\begin{enumerate}
\item The metric $g$ is bi-invariant if and only if 
$\nabla=\nabla^{(0)}$.
\item $S^{(-)}=\nabla^{(-)}-\nabla$ is a 
homogeneous Riemannian structure of $G$.
\end{enumerate}
\end{proposition}
The homogeneous Riemannian structure
$S^{(-)}=\nabla^{(-)}-\nabla$ is called the 
\emph{canonical homogeneous structure} 
in \cite{CalLo,CFG}.
\section{Moving frames}\label{sec:3}
For later use, here we collect fundamental equations 
of moving frames on Riemannian $3$-manifolds.
\subsection{Connection forms}
Let $(M,g)$ be a Riemannian $3$-manifold. Take a local orthonormal frame field 
$\mathcal{E}=\{e_1,e_2,e_3\}$. Denote by
$\Theta=(\theta^1,\theta^2,\theta^3)$ 
the orthonormal coframe field 
metrically dual to $\{e_1,e_2,e_3\}$. 
We regard $\Theta$ as a vector valued $1$-form 
\[
\Theta=\left(
\begin{array}{c}
\theta^1\\
\theta^2\\
\theta^3
\end{array}
\right).
\]
Since the 
Levi-Civita connection is torsion free, 
the following 
\emph{first structure equation}
\[
\mathrm{d}\Theta+\omega\wedge \Theta=0.
\]
holds.
The $\mathfrak{so}(3)$-valued $1$-form  
\[
\mathbf{\omega}=
\left(
\begin{array}{ccc}
0 & \omega_{2}^{\>\,1} & \omega_{3}^{\>\,1}
\\
-\omega_{2}^{\>\,1} & 0 & \omega_{3}^{\>\,2}\\
- \omega_{3}^{\>\,1} & -\omega_{3}^{\>\,2} & 0 
\end{array}
\right)
\]
determined by the first
structure equation is called the 
\emph{connection form}.  
A component $\omega_{j}^{\>i}$ of $\omega$ is 
called a connection $1$-form. 
The first structure equation is 
the differential system:
\[
\mathrm{d}\theta^{i}+\sum_{j=1}^{3}\omega_{j}^{\>\,i}\wedge \theta^{j}=0.
\]
The connection coefficients 
$\{\varGamma_{jk}^{\,\>i}\}$ of 
the Levi-Civita connection $\nabla$ is 
relative to $\mathcal{E}$ is defined by
\[
\nabla_{e_i}e_{j}=\sum_{k=1}^{3}\varGamma_{ij}^{\,\,k}e_k.
\]
Then the connection $1$-forms are related to  
connection coefficients by
\[
\omega_{j}^{\>k}=
\sum_{\ell=1}^{3}\varGamma_{\ell j}^{\,\,k}\,\theta^{\ell}.
\]
Hence we obtain
\begin{equation}\label{eq:LC-omega}
g(\nabla_{X}e_i,e_j)=
\omega_{i}^{\>j}(X).
\end{equation}
Thus 
\[
\omega_{i}^{\>\>j}=-\omega_{j}^{\>\>i}.
\]
\begin{remark}
{\rm Tricerri and Vanhecke \cite{TV} used the convention:
\[
g(\nabla_{X}e_i,e_j)=
\omega_{ij}(X).
\]
}
\end{remark}
\subsection{Curvature forms}
Next, 
the $\mathfrak{so}(3)$-valued 
$2$-form $\varOmega=(\varOmega_{j}^{\>\,i})$ defined by
\[
\varOmega=\mathrm{d}\omega+\omega\wedge\omega
\]
is called the 
\emph{curvature form} relative to $\Theta$. 
This formula is called the 
\emph{second structure equation}. 
The components $\varOmega_{j}^{\>\>i}$ are called 
\emph{curvature} $2$-\emph{forms}. The second structure equation is the 
differential system:
\[
\varOmega_{j}^{\>\,i}=
\mathrm{d}\omega_{j}^{\>\,i}+\sum_{k=1}^{3}
\omega_{k}^{\>\,i}\wedge
\omega_{j}^{\>\,k}.
\]
One can see 
that 
\[
R(X,Y)Z=2\sum_{i=1}^{3}\varOmega_{i}^{\>\>j}(X,Y)e_{j}.
\]
If we express the Riemannian curvature $R$ as
\[
R(e_k,e_{\ell})e_i=\sum_{j=1}^{3}
R_{ik\ell}^{\>\>j}\,e_{j},
\]
and set
\[
R_{ijk\ell}
=g(R(e_k,e_{\ell})e_i,e_j)=R_{ik\ell}^{\>\>j},
\quad \varOmega_{ij}:=\varOmega_{i}^{\>\>j},
\]
then we obtain
\[
\varOmega_{ij}
=\frac{1}{2}
\sum_{k,\ell=1}^{3} R_{ijk\ell}\,\theta^k\wedge\theta^{\ell}.
\]

\subsection{Homogeneous Riemannian structures}
Let $S_{\flat}$ be a tensor field of 
type $(0,3)$ on a Riemannian $3$-manifold $(M,g)$ 
satisfying
\begin{equation}\label{eq:3.2}
S_{\flat}(X,Y,Z)+S_{\flat}(X,Z,Y)=0.
\end{equation}
Then 
$S_{\flat}$ is expressed as
\begin{multline}\label{eq:3.3}
S_{\flat}(X,Y,Z)
=2S_{\flat}(X,e_1,e_2)(\theta^1\wedge \theta^2)(Y,Z) \\
+2S_{\flat}(X,e_2,e_3)(\theta^2\wedge \theta^3)(Y,Z)
+2S_{\flat}(X,e_3,e_1)(\theta^3\wedge \theta^1)(Y,Z).
\end{multline}
We define a tensor field $S$ by
\begin{equation}\label{eq:3.4}
S_{\flat}(X,Y,Z)=g(S(X)Y,Z).
\end{equation}
By using $S$, we define a linear connection
$\tilde{\nabla}$ by $\tilde{\nabla}=\nabla+S$. 
Assume that $S$ is a homogeneous Riemannian structure, then 
$\tilde{\nabla}$ satisfies 
$\tilde{\nabla}R=0$. 
Since $M$ is $3$-dimensional, 
this condition is 
equivalent to $\tilde{\nabla}\mathrm{Ric}=0$, 
\textit{i.e.},
\begin{equation}\label{eq:5.5}
(\nabla_{X}\mathrm{Ric})(Y,Z)=\mathrm{Ric}(S(X)Y,Z)+\mathrm{Ric}(Y,S(X)Z).
\end{equation} 
Let us assume that $\{e_1,e_2,e_3\}$ diagonalizes 
the Ricci tensor field;
\[
\mathrm{Ric}(e_i,e_i)=\rho_i\,\delta_{ij},
\quad 
i,j=1,2,3,
\] 
where $\delta_{ij}$ is the Kronecker's delta.
The components $\rho_1$, $\rho_2$ and 
$\rho_3$ are called 
the \emph{principal Ricci curvatures}.

In terms of $\{e_1,e_2,e_3\}$, the parallelism equation 
\eqref{eq:5.5} is expressed as

\[
-\sum_{j,k=1}^{3}\omega_{i}^{\>k}(X)\delta_{k}^{\>j}\rho_{j}
-\sum_{i,k=1}^{3}\omega_{j}^{\>k}(X)\delta_{k}^{\>i}\rho_{i}
=\sum_{j,k=1}^{3}S_{\flat}(X,e_i,e_k)\delta^{kj}\rho_{j}
+\sum_{i,k=1}^{3}S_{\flat}(X,e_j,e_k)\delta^{ki}\rho_{i}.
\]
From these equations we deduce that
\begin{equation}\label{eq:3.6}
(\rho_1-\rho_2)\omega_{2}^{\>1}(X)=(\rho_1-\rho_2)S_{\flat}(X,e_1,e_2),
\end{equation}
\begin{equation}\label{eq:3.7}
-(\rho_1-\rho_3)\omega_{3}^{\>1}(X)=(\rho_1-\rho_3)S_{\flat}(X,e_3,e_1),
\end{equation}
\begin{equation}\label{eq:3.8}
(\rho_2-\rho_3)\omega_{3}^{\>2}(X)=(\rho_2-\rho_3)S_{\flat}(X,e_2,e_3).
\end{equation}
Hence we obtain the following result.
\begin{lemma}\label{lem:3.1}
Let $(M,g)$ be a Riemannian $3$-manifold whose 
principal Ricci curvatures are distinct. 
Take a local orthonormal frame field $\{e_1,e_2,e_3\}$ 
which diagonalizes the Ricci tensor field. 
If a tensor field $S$ of type $(1,2)$ on $M$ is 
a homogeneous 
Riemannian structure, then its covariant form $S_{\flat}$ satisfies 
\[
S_{\flat}=2\omega_{2}^{\>1}\otimes(\theta^1\wedge\theta^2)
+2\omega_{3}^{\>2}\otimes(\theta^2\wedge\theta^3)
+2\omega_{1}^{\>3}\otimes(\theta^3\wedge\theta^1),
\]
where $\{\theta^1,\theta^2,\theta^3\}$ is the dual 
coframe field of $\{e_1,e_2,e_3\}$ and 
$\{\omega_{j}^{\>i}\}$ are connection $1$-forms 
of $M$ relative to $\{\theta^1,\theta^2,\theta^3\}$.
\end{lemma}

\begin{remark}[Pseudo-symmetry]
{\rm Let $(M,g)$ be a Riemannian $3$-manifold with 
principal Ricci curvatures $\{\rho_1,\rho_2,\rho_3\}$. 
Then 
\begin{itemize}
\item $(M,g)$ is of constant curvature if and only if 
$\rho_1=\rho_2=\rho_3$. In such a case, the 
principal Ricci curvature is constant.
\item $(M,g)$ is said to be \emph{pseudo-symmetric} (in 
the sense of Deszcz \cite{Deszcz}) if at least two 
principal Ricci curvatures coincide.
\end{itemize}
}
\end{remark}

\section{Three dimensional Lie groups}\label{sec:4}
\subsection{Unimodularity}
Let $G$ be a Lie group with a Lie algebra
$\mathfrak{g}$ and a left invariant Riemannian metric $g$. 
We denote by 
$\langle\cdot,\cdot\rangle=g_{\mathsf e}$ the inner product 
of $\mathfrak{g}$ induced from $g$. 
Then the Levi-Civita connection
$\nabla$ of $g$ 
is described by the \emph{Koszul formula}:
\[
2\langle \nabla_{X}Y,Z \rangle=-\langle
X,[Y,Z]\rangle+\langle
Y,[Z,X]\rangle+
\langle 
Z,[X,Y]
\rangle,
\quad 
X,Y,Z \in \mathfrak{g}.
\] 
Let us define a symmetric bilinear map $U:\mathfrak{g}\times
\mathfrak{g}\to\mathfrak{g}$ by
\begin{equation}\label{obstrction}
2\langle U(X,Y),Z
\rangle=\langle X,[Z,Y]\rangle+
\langle Y,[Z,X]\rangle
\end{equation}
and call it the \emph{bi-invariance obstruction} of 
$(G,g)$.
One can see that the metric $g$ is right-invariant if and only if $U=0$.

A Lie group $G$ is said to be \emph{unimodular} if its left
invariant Haar measure is right invariant. Milnor gave an
infinitesimal reformulation of unimodularity for $3$-dimensional Lie
groups. We recall it briefly here.
\medskip

Let $\mathfrak{g}$ be a $3$-dimensional oriented Lie algebra with an
inner product $\langle\cdot,\cdot\rangle$. 
Denote by $\times$ the \emph{vector
product operation} of the 
oriented inner product space $(\mathfrak{g},\langle\cdot,\cdot
\rangle)$. 
The vector product operation is a 
skew-symmetric bilinear map 
$\times:\mathfrak{g}\times \mathfrak{g}\to
\mathfrak{g}$ which is uniquely
determined by the following conditions:
\begin{itemize}
\item[(i)]
$\langle X,X\times Y\rangle=\langle Y,X\times Y\rangle=0$,
\item[(ii)] $\langle X\times Y,X\times Y\rangle=\langle X,X\rangle
\langle Y,Y\rangle-\langle X,Y\rangle^2$,
\item[(iii)] if $X$ and $Y$ are linearly independent, then
$\det(X,Y,X\times Y)>0$
\end{itemize}
for all $X,Y\in \mathfrak{g}$. On the other hand, the Lie-bracket
$[\cdot,\cdot]:\mathfrak{g}\times \mathfrak{g}\to\mathfrak{g}$ is a
skew-symmetric bilinear map. Comparing these two operations, we get
a linear endomorphism $L_{\mathfrak{g}}$ which is uniquely
determined by the formula

\begin{equation*}\label{vectorproduct}
[X,Y]=L_{\mathfrak{g}}(X\times Y),\ \ X,Y \in \mathfrak{g}.
\end{equation*}
Now let $G$ be an oriented $3$-dimensional Lie group equipped with a
left invariant Riemannian metric. Then the metric induces an inner
product on the Lie algebra $\mathfrak{g}$. With respect to the
orientation on $\mathfrak{g}$ induced from $G$, the endomorphism
field $L_{\mathfrak{g}}$ is uniquely determined. The unimodularity
of $G$ is characterised as follows.
\begin{proposition}{\rm(\cite{Milnor})}
Let $G$ be an oriented $3$-dimensional Lie group with a left
invariant Riemannian metric. Then $G$ is unimodular if and only if
the endomorphism $L_{\mathfrak{g}}$ is self-adjoint with respect to
the metric.
\end{proposition}
The \emph{unimodular kernel} $\mathfrak{u}$ of $\mathfrak{g}$ is defined by
\[
\mathfrak{u}=\{X \in \mathfrak{g}
\ \vert \
\mathrm{tr}\> \mathrm{ad}(X)=0\}.
\]
Here $\mathrm{ad}:\mathfrak{g}\to \mathrm{End}(\mathfrak{g})$ is 
a homomorphism defined by
\[
\mathrm{ad}(X)Y=[X,Y].
\]
One can see that $\mathfrak{u}$ is an ideal of $\mathfrak{g}$ which contains 
the ideal $[\mathfrak{g},\mathfrak{g}]$. 
One can see that $G$ is unimodular if and only 
if $\mathfrak{u}=\mathfrak{g}$.

\subsection{Unimodular Lie groups}

Let $G$ be a $3$-dimensional unimodular Lie group with a left
invariant metric $g$. Then there exists an
orthonormal basis $\mathcal{E}=\{e_1,e_2,e_3\}$ of the Lie algebra
$\mathfrak{g}$ (called a \emph{unimodular basis}) such that
\begin{equation}\label{basis}
[e_1,e_2]=c_{3}e_{3},\quad  [e_2,e_3]=c_{1}e_{1},\quad
[e_3,e_1]=c_{2}e_{2}, \qquad c_{i}\in \mathbb{R}.
\end{equation}

Three-dimensional unimodular Lie groups are classified by Milnor as
follows (up to orientation preserving numeration of $e_1$, $e_2$, $e_3$ and changing 
signs):

\medskip
\begin{table}[h]
\centering
{\renewcommand\arraystretch{1.35}
\begin{tabular}[h]{|c|c|c|}
\hline
   Signature of $(c_1,c_2,c_3)$
& Simply connected Lie group
& Property \\
  \hline
%%%%%%%%%%%%%%%%%%%%%%%%%%%%%%%%
  $(+,+,+)$
& $\mathrm{SU}(2)$
& compact and simple \\
%%%%%%%%%%%%%%%%%%%%%%%%%%%%%%
$(+,+,-)$ & $\widetilde{\mathrm{SL}}_{2}\mathbb{R}$ & non-compact
and simple
\\
%%%%%%%%%%%%%%%%%%%%%%%%%%%%%%%
$(+,+,0)$ & $\widetilde{\mathrm{SE}}(2)$ & solvable
\\
%%%%%%%%%%%%%%%%%%%%%%%%%%%%%%%
$(+,-,0)$ & $\mathrm{SE}(1,1)$ & solvable
\\
%%%%%%%%%%%%%%%%%%%%%%%%%%%%%
$(+,0,0)$ 
& Heisenberg group $\mathrm{Nil}_3$ & nilpotent
\\
%%%%%%%%%%%%%%%%%%%%%%%%%%%%
$(0,0,0)$ & $(\mathbb{R}^3,+)$ & Abelian
\\
\hline
\end{tabular}
}
\smallskip
\caption{Three dimensional unimodular Lie groups}
\label{table2}
\end{table}

To describe the Levi-Civita connection $\nabla$ of $G$, we introduce 
the following constants:
\[
\mu_{i}=\frac{1}{2}(c_{1}+c_{2}+c_{3})-c_{i},\quad i=1,2,3.
\]
\begin{proposition}\label{prp:unilevi}
The Levi-Civita connection is given by
\[
\begin{array}{ccc}
\nabla_{e_1}e_{1}=0, & \nabla_{e_1}e_{2}=\mu_{1}e_{3}, & \nabla_{e_1}e_{3}=-\mu_{1}e_{2}\\
\nabla_{e_2}e_{1}=-\mu_{2}e_{3}, & \nabla_{e_2}e_{2}=0, & \nabla_{e_2}e_{3}=\mu_{2}e_{1}\\
\nabla_{e_3}e_{1}=\mu_{3}e_2, & \nabla_{e_3}e_{2}=-\mu_{3}e_{1} & \nabla_{e_3}e_{3}=0.
\end{array}
\]
The Riemannian curvature $R$ is given by
\[
R(e_1,e_2)e_1=(\mu_{1}\mu_{2}-c_{3}\mu_{3})e_{2},\ \ 
R(e_1,e_2)e_2=-(\mu_{1}\mu_{2}-c_{3}\mu_{3})e_{1},\ \ 
\]
\[
R(e_2,e_3)e_2=(\mu_{2}\mu_{3}-c_{1}\mu_{1})e_{3},\ \ 
R(e_2,e_3)e_3=-(\mu_{2}\mu_{3}-c_{1}\mu_{1})e_{2},\ \ 
\]
\[
R(e_1,e_3)e_1=(\mu_{3}\mu_{1}-c_{2}\mu_{2})e_{3},\ \ 
R(e_1,e_3)e_3=-(\mu_{3}\mu_{1}-c_{2}\mu_{2})e_{1}. 
\]
The basis $\{e_1,e_2,e_3\}$ diagonalizes the Ricci tensor field.
The principal Ricci curvatures are given by
\[
\rho_{1}=2\mu_{2}\mu_{3},\quad
\rho_{2}=2\mu_{1}\mu_{3},\quad
\rho_{3}=2\mu_{1}\mu_{2}.
\]
The bi-invariance obstruction $U$ is given by
\[
U(e_1,e_2)=\frac{1}{2}(-c_1+c_2)e_3,\ \
U(e_1,e_3)=\frac{1}{2}(c_1-c_3)e_2,\ \
U(e_2,e_3)=\frac{1}{2}(-c_2+c_3)e_1.
\]
\end{proposition}
The Cartan-Schouten's $(-)$-connection 
is described as $\nabla^{(-)}=\nabla+S^{(-)}$, where
\begin{equation}\label{eq:minusconnection}
S_{\flat}^{(-)}
=
-2\mu_{3}\theta^{3}\otimes
(\theta^1\wedge \theta^2)
-2\mu_{1}\theta^{1}\otimes(\theta^2\wedge \theta^3)
-2\mu_{2}\theta^{2}\otimes
(\theta^3\wedge \theta^1).
\end{equation}
The following fact is well known 
(see \textit{e.g}., \cite{MP}):
\begin{proposition}
Let $G$ be a $3$-dimensional 
unimodular Lie group with unimodular basis 
$\mathcal{E}=\{e_1,e_2,e_3\}$. If all the structure 
constants of $\mathcal{E}$ are distinct, 
then the isometry group of $G$ is $3$-dimensional.
\end{proposition}
Tsukada \cite{Tsukada} classified 
totally geodesics surfaces in 
$3$-dimensional unimodular Lie groups 
(see Table \ref{table3}).

\medskip
\begin{table}[h]
\centering
{\renewcommand\arraystretch{1.35}
\begin{tabular}[h]{|c|c|c|c|}
\hline
 Simply connected Lie group
 & Metric 
 & totally geodesic surfaces
& Property \\
  \hline
%%%%%%%%%%%%%%%%%%%%%%%%%%%%%%%%
$\mathrm{SU}(2)$
&
$c_1=c_2=c_3$
& All planes
& (1)
\\
{}
&
other cases
& $\varnothing$
& 
\\
\hline
%%%%%%%%%%%%%%%%%%%%%%%%%
{}
&
$c_1>c_2-c_3$
& $\varnothing$
& 
\\
$\widetilde{\mathrm{SL}}_{2}\mathbb{R}$
&
$c_1=c_2-c_3$ 
&
$\mathbb{R}(\cos\theta e_1+\sin\theta e_3)
\oplus\mathbb{R}e_2$
& (2)
\\
{}
&
$c_1<c_2-c_3$ 
&
$\varnothing$
& 
\\
%%%%%%%%%%%%%%%%%%%%%%%%%%
\hline
$\widetilde{\mathrm{SE}}(2)$
&
$c_1>c_2$
& $\varnothing$
& 
\\
{}
&
%$c_2=c_3$ 
$c_1=c_2$
&
All planes
& (3)
\\
%%%%%%%%%%%%%%%%%%%%%%%%%%%%%%
\hline
$\mathrm{SE}(1,1)$
&
$c_1\not=-c_2$
& $\varnothing$
& 
\\
{}
&
$c_1=-c_2$ 
&
$\mathbb{R}(e_1\pm e_2)\oplus\mathbb{R}e_3$
& (4)
\\
\hline
%%%%%%%%%%%%%%%%%%%%%%%%%%%%%
$\mathrm{Nil}_3$ 
& {}
& $\varnothing$
&
\\
\hline
%%%%%%%%%%%%%%%%%%%%%%%%%%%%
$\mathbb{R}^3$ 
& 
& All planes
& (5)
\\
\hline
\end{tabular}
}
\caption{Totally geodesic surfaces}
\label{table3}
\end{table}
\medskip

\begin{remark}[Properties in Table 3]
\mbox{}
{\rm 
\begin{enumerate}
\renewcommand{\labelenumi}{(\arabic{enumi})}
\item The metric is bi-invariant and of constant positive curvature.
\item These linear subspaces are invariant under the action of 
$\mathrm{Aut}(\mathfrak{g})\cap\mathrm{O}(\mathfrak{g})$.
\item The meric is flat.
\item These linear subspaces are invariant under the action of 
$\mathrm{Aut}(\mathfrak{g})\cap\mathrm{O}(\mathfrak{g})$.
\item The meric is flat.
\end{enumerate}
}
\end{remark}

More generally, parallel surfaces 
in $3$-dimensional Riemannian groups are 
classified in \cite{IV1,IV2}. 
Grassmann geometry of surfaces of
orbit type are developed in 
\cite{IKN,IN1,IN2,IN,IN3,Kuwabara}.

\subsection{Non-unimodular Lie groups}
Let $G$ be a $3$-dimensional \emph{non-unimodular} Lie group with a left 
invariant metric $g$. 
Then we can take an orthonormal basis
$\{e_1,e_2,e_3\}$ of the Lie algebra 
$\mathfrak{g}$ of $G$ such that
\begin{enumerate}
\item $\langle e_{1}, X\rangle=0,\ X \in \mathfrak{u}$,
\item $\langle [e_1,e_2], [e_1,e_3]\rangle=0$,
\end{enumerate} 
where $\mathfrak{u}$ is the unimodular kernel of $\mathfrak{g}$ as before.
Then the commutation relations of the basis are given by
\begin{equation}
\label{Lie-bra-n-u}
[e_1,e_2]=a_{11} e_{2}+a_{21} e_{3},\
[e_2,e_3]=0,\ 
[e_1,e_3]=a_{12} e_2+a_{22} e_3,
\end{equation}
with $a_{11}+a_{22}\not=0$ and $a_{11}a_{21}+a_{12}a_{22}=0$. 
We denote by $A$ the $2\times 2$ matrix whose $(i,j)$-entry is $a_{ij}$. 
Under a suitable homothetic change of the metric, we may assume that
$a_{11}+a_{22}=2$. 
Moreover, we may assume that $a_{11}a_{22}+a_{12}a_{21}=0$ \cite[\S 2.5]{MP}. 
Then the constants $a_{11}$,
$a_{12}$, $a_{21}$ and $a_{22}$ are represented
as 
\begin{equation}
\label{n-u-coeff}
a_{11}=1+\alpha,\ a_{21}=(1+\alpha)\beta,\
a_{12}=-(1-\alpha)\beta,\
a_{22}=1-\alpha. 
\end{equation}
If necessarily, 
by changing the sign of $e_1$, $e_2$ and 
$e_3$, we may assume that the constants $\alpha$ and 
$\beta$ satisfy the condition $\alpha$, $\beta\geq 0$.
We note that for the case that $\alpha=\beta=0$, $G$ is of constant negative 
curvature (see Example \ref{eg:4.1}). 
We refer $(\alpha,\beta)$ as the structure constants 
of the non-unimodular Lie algebra $\mathfrak{g}$.

Under this normalization, 
the Levi-Civita connection of $G$ is given by the following table:
\begin{proposition}\label{prp:nonunilevi}
\[
\begin{array}{ccc}
\nabla_{e_1}e_{1}=0, & \nabla_{e_1}e_{2}=\beta e_{3}, & \nabla_{e_1}e_{3}=-\beta e_{2},\\
\nabla_{e_2}e_{1}=-(1+\alpha)e_{2}-\alpha\beta e_{3}, 
& \nabla_{e_2}e_{2}=(1+\alpha) e_1, & 
\nabla_{e_2}e_{3}=\alpha\beta e_{1},\\
\nabla_{e_3}e_{1}=-\alpha\beta e_{2}
-(1-\alpha)e_{3}, 
& \nabla_{e_3}e_{2}=\alpha\beta e_{1}, & \nabla_{e_3}e_{3}=(1-\alpha)e_1.
\end{array}
\]
The Riemannian curvature $R$ is given by
\begin{align*}
R(e_1,e_2)e_1=&\{\alpha\beta^{2}+(1+\alpha)^{2}+\alpha\beta^{2}(1+\alpha)\}e_{2},
\\
R(e_1,e_2)e_2=&-\{\alpha\beta^{2}+(1+\alpha)^{2}+\alpha\beta^{2}(1+\alpha)\}e_{1},
\\
R(e_1,e_3)e_1=&-\{\alpha\beta^{2}-(1-\alpha)^{2}+\alpha\beta^{2}(1-\alpha)\}e_{3},
\\
R(e_1,e_3)e_3=&\{\alpha\beta^{2}-(1-\alpha)^{2}+\alpha\beta^{2}(1-\alpha)\}e_{1},
\\
R(e_2,e_3)e_2=&\{1-\alpha^{2}(1+\beta^{2})\}e_{3},
\\
R(e_2,e_3)e_3=&-\{1-\alpha^{2}(1+\beta^{2})\}e_{2}.
\end{align*}
The basis $\{e_1,e_2,e_3\}$ diagonalizes the Ricci tensor field.
The principal Ricci curvatures are given by
\[
\rho_{1}=-2\{1+\alpha^2(1+\beta^2)\}<-2,\ \ 
\rho_{2}=-2\{1+\alpha(1+\beta^2)\}<-2,\ \ 
\rho_{3}=-2\{1-\alpha(1+\beta^2)\}.
\]
The scalar curvature is 
\[
-2\{3+\alpha^2(1+\beta^2)\}<0.
\]
\end{proposition}
Non-unimodular Lie algebras $\mathfrak{g}=\mathfrak{g}(\alpha,\beta)$ 
are classified by the 
\emph{Milnor invariant} $\mathcal{D}:=\det A=(1-\alpha^2)(1+\beta^2)$.
More precisely we have the following fact (\textit{cf} \cite{MP, Milnor}).
\begin{proposition}
For any pair of $(\alpha,\beta)$ and 
$(\alpha^\prime,\beta^\prime)$ which are not $(0,0)$, two Lie algebras 
$\mathfrak{g}(\alpha,\beta)$ and $\mathfrak{g}(\alpha^\prime,\beta^\prime)$ are 
isomorphic if and only if their Milnor invariants $\mathcal{D}$ and 
$\mathcal{D}^\prime$ agree.
\end{proposition}

\begin{remark}{\rm 
Tasaki and Umehara introduced an invariant of 
$3$-dimensional Lie algebras equipped with inner products \cite{TU}.
Their invariant $\chi(\mathfrak{g})$ for the non-unimodular 
Lie algebra $\mathfrak{g}=\mathfrak{g}(\alpha,\beta)$ is $4/\mathcal{D}$. 
Note that in case $\mathcal{D}=0$, $\chi(\mathfrak{g})$ is regarded as $\infty$. 
}
\end{remark}
The characteristic equation for the matrix $A$ is $\lambda^2-2\lambda+\mathcal{D}=0$. 
Thus the eigenvalues of $A$ are $1\pm\sqrt{1-\mathcal{D}}$. 
Roughly speaking, the moduli space of 
non-unimodular Lie algebras $\mathfrak{g}(\alpha,\beta)$ may be 
decomposed into three kinds of types which are determined by the 
conditions $\mathcal{D}>1$, $\mathcal{D}=1$ and $\mathcal{D}<1$.

\begin{remark}
{\rm The sign of $\rho_3$ can be positive, null or negative.
In fact, if $\mathcal{D}<0$, then $\rho_3>0$. 
If $\mathcal{D}\geq 0$, then $\rho_3
=0$ is possible. In case $\mathcal{D}>0$, there exists a 
metric of strictly negative curvature. Moreover if 
$\mathcal{D}>1$, there exists a metric of constant negative curvature.
}
\end{remark}
Let us orient $\mathfrak{g}(\alpha,\beta)$ so that 
$\{e_1,e_2,e_3\}$ is positive. Then 
the cross product $\times$ with respect to 
this orientation defines a linear operator $L$ 
on $\mathfrak{g}(\alpha,\beta)$ by
\[
L(u\times v)=[u,v], \quad u,v \in
\mathfrak{g}(\alpha,\beta).
\]
Then we obtain
\[
L(e_1)=0,\quad 
L(e_2)= (1-\alpha)\beta e_2+(\alpha-1)e_3,\quad 
L(e_3)=(1+\alpha)e_2+(1+\alpha)\beta e_3.
\]
Since $\mathfrak{g}$ is non-unimodular, $L$ is \emph{not} self-adjoint.

\medskip

The simply connected Lie 
group $\widetilde{G}=\widetilde{G}(\alpha,\beta)$ corresponding to  
the non-unimodular Lie algebra $\mathfrak{g}(\alpha,\beta)$ 
is given explicitly by \cite{IN}:
\[
\widetilde{G}(\alpha,\beta)=
\left\{
\left.
\left(
\begin{array}{cccc}
1 & 0 & 0 & x\\
0 & \alpha_{11}(x) & \alpha_{12}(x) & y\\
0 & \alpha_{21}(x) & \alpha_{22}(x) & z\\
0 & 0 & 0 &1
\end{array}
\right)
\
\right
\vert
\
x,y,z \in \mathbb{R}
\right\},
\]
where 
\[
\left(
\begin{array}{cc}
\alpha_{11}(x) & \alpha_{12}(x)\\
\alpha_{21}(x) & \alpha_{22}(x)
\end{array}
\right)=
\exp\left[
x
\left(
\begin{array}{cc}
1+\alpha & -(1-\alpha)\beta\\
(1+\alpha) & 1-\alpha
\end{array}
\right)
\right].
\]
This shows that $\widetilde{G}(\alpha,\beta)$ is the semi-direct product
$\mathbb{R}\ltimes_{e^{xA}}\mathbb{R}^2$ 
with multiplication
\begin{equation*}
\label{semi-direct}
(x,y,z)\cdot (x^\prime,y^\prime,z^\prime)=
(x+x^\prime,y+\alpha_{11}(x)y^\prime+\alpha_{12}(x)z^\prime,
z+\alpha_{21}(x)y^\prime+\alpha_{22}(x)z^\prime).
\end{equation*}
The Lie algebra of $\widetilde{G}(\alpha,\beta)$ is 
spanned by the basis
\[
e_{1}=
\left(
\begin{array}{cccc}
0 & 0 & 0 & 1\\
0 & 1+\alpha & -(1-\alpha)\beta & 0\\
0 & (1+\alpha)\beta & 1-\alpha & 0\\
0 & 0 & 0 &0
\end{array}
\right),
\ \
e_{2}=
\left(
\begin{array}{cccc}
0 & 0 & 0 & 0\\
0 & 0 & 0 & 1\\
0 & 0 & 0 & 0\\
0 & 0 & 0 &0
\end{array}
\right),
\ \
e_{3}=
\left(
\begin{array}{cccc}
0 & 0 & 0 & 0\\
0 & 0 & 0 & 0\\
0 & 0 & 0 & 1\\
0 & 0 & 0 &0
\end{array}
\right).
\]
This basis satisfies the commutation relation
\[
[e_1,e_2]=(1+\alpha)(e_{2}+\beta e_{3}),\ \
[e_2,e_3]=0,\ \ 
[e_3,e_1]=(1-\alpha)(\beta e_2-e_{3}).
\]
Thus the Lie algebra of $\widetilde{G}(\alpha,\beta)$ 
is the non-unimodular Lie algebra 
$\mathfrak{g}=\mathfrak{g}(\alpha,\beta)$. 
The left invariant vector fields corresponding to 
$e_1$, $e_2$ and $e_3$ are
\[
e_1=\frac{\partial}{\partial x},\ \ 
e_2=\alpha_{11}(x)\frac{\partial}{\partial y}+\alpha_{12}(x)\frac{\partial}{\partial z},
\ \
e_3=\alpha_{21}(x)\frac{\partial}{\partial y}+\alpha_{22}(x)\frac{\partial}{\partial z}.
\]
\begin{example}[$\alpha=0$, $\mathcal{D}\geq 1$]\label{eg:4.1}
{\rm The simply connected Lie group 
$\widetilde{G}(0,\beta)$ is isometric to the hyperbolic $3$-space 
$\mathbb{H}^3(-1)$ of curvature $-1$ and given 
explicitly by
\[
\widetilde{G}(0,\beta)=
\left\{
\left.
\left(
\begin{array}{cccc}
1 & 0 & 0 & x\\
0 & e^{x}\cos(\beta x) & -e^{x}\sin(\beta x) & y\\
0 & e^{x}\sin(\beta x) & e^{x}\cos(\beta x) & z\\
0 & 0 & 0 &1
\end{array}
\right)
\>
\right
\vert
\>
x,y,z \in \mathbb{R}
\right\}. 
\]
The left invariant metric is 
$dx^2+e^{-2x}(dy^2+dz^2)$. 
Thus $\widetilde{G}(0,\beta)$
is the warped product model of $\mathbb{H}^{3}(-1)$.
Indeed, by setting $w=e^{x}$, then the left invariant metric of $\widetilde{G}$ 
is rewritten as the Poincar{\'e} metric
\[
\frac{dy^2+dz^2+dw^2}{w^2}.
\]
The Milnor invariant of $\widetilde{G}(0,\beta)$ is 
$\mathcal{D}=1+\beta^2\geq 1$.
}
\end{example}

\begin{example}[$\beta=0$, $\mathcal{D}\leq1$]\label{eg:4.2}{\rm
For each $\alpha\geq 0$, $\widetilde{G}(\alpha,0)$ is given by:
\[
\widetilde{G}(\alpha,0)=\left\{
\left(
\left.
\begin{array}{cccc}
1 & 0 & 0 & x\\
0 & e^{(1+\alpha)x} & 0 & y\\
0 & 0 & e^{(1-\alpha)x} & z\\
0 & 0 & 0 &1
\end{array}
\right)
\
\right
\vert
\
x,y,z \in \mathbb{R}
\right\}.
\]
The left invariant Riemannian metric is given explicitly by
\[
dx^{2}+e^{-2(1+\alpha)x}dy^{2}+e^{-2(1-\alpha)x}dz^{2}.
\]
The Milnor invariant is $\mathcal{D}=1-\alpha^2\leq 1$.

\noindent
The Lie group $\widetilde{G}(\alpha,0)$ 
is realized as a warped product 
$\mathbb{H}^2(-(1+\alpha)^2)\times_{f_{\alpha}}\mathbb{R}$ of 
the hyperbolic plane $\mathbb{H}^{2}(-(1+\alpha)^2)$
of curvature $-(1+\alpha)^2$ and the real line $\mathbb{R}$ with 
warping function $f_{\alpha}(x)=e^{(\alpha-1)x}$. 
In fact, via the coordinate change 
$(u,v)=(\,(1+\alpha)y,e^{(1+\alpha)x})$, the metric is rewritten as
\[
\frac{du^2+dv^2}{(1+\alpha)^2v^2}+f_{\alpha}(u,v)^2\, dz^2, \ \ f_{\alpha}(u,v)=
\exp\left(
\frac{\alpha-1}{\alpha+1}\log v
\right).
\]
Here we observe locally symmetric examples:
\begin{itemize}
\item If $\alpha=0$ then $\widetilde{G}(0,0)$ is a warped product model of 
hyperbolic 3-space $\mathbb{H}^{3}(-1)$. 
\item If $\alpha=1$ then $\widetilde{G}(1,0)$ is 
isometric to the Riemannian product $\mathbb{H}^{2}(-4)\times \mathbb{R}$. 
\end{itemize}
}
\end{example}

\begin{example}[$\alpha=1$, $\mathcal{D}=0$]\label{eg:4.3}
{\rm
Assume that $\alpha=1$. Then $\widetilde{G}(1,\beta)$ is given explicitly by
\[
\widetilde{G}(1,\beta)=
\left\{
\left.
\left(
\begin{array}{cccc}
1 & 0 & 0 & x\\
0 & e^{2x} & 0 & y\\
0 & \beta(e^{2x}-1) & 1 & z\\
0 & 0 & 0 &1
\end{array}
\right)
\
\right
\vert
\
x,y,z \in \mathbb{R}
\right\}.
\]
The left invariant metric is 
\[
dx^2+\{e^{-4x}+\beta^2(1-e^{-2x})^2\}dy^2-2\beta(1-e^{-2x})dydz+dz^2.
\]
The non-unimodular Lie group $\widetilde{G}(1,\beta)$ has sectional curvatures
\[
K_{12}=-3\beta^{2}-4,\ 
\ 
K_{13}=K_{23}=\beta^{2},
\]
where $K_{ij}$ ($i\not=j$) denote the 
sectional curvatures of the planes spanned
by vectors $e_i$ and $e_j$. 
The principal Ricci curvatures are 
\[
\rho_1=\rho_2=-2(\beta^2+2),
\quad 
\rho_3=2\beta^2\geq 0.
\]
Hence $\mathrm{Ric}$ has the form
\[
\mathrm{Ric}=-2(\beta^2+2)g
+4(\beta^2+1)\theta^3\otimes\theta^3.
\]
One can check that $\widetilde{G}(1,\beta)$ is isometric to
the so-called 
\emph{Bianchi-Cartan-Vranceanu space} $M^{3}(-4,\beta)$
with $4$-dimensional isometry group and 
isotropy subgroup $\mathrm{SO}(2)$:
\[
M^{3}(-4,\beta)=
\left(
\{(u,v,w)\in \mathbb{R}^{3}
\
\vert 
\
u^2+v^2<1\},g_{_\beta}
\right),
\]
with metric 
\[
g_{_\beta}=\frac{du^2+dv^2}{(1-u^2-v^2)^{2}}+
\left(
dw+ \frac{\beta(vdu-udv)}{1-u^2-v^2}
\right)^{2},\ \ \beta\geq 0.
\]
The family $\{\widetilde{G}(1,\beta)\}_{\beta\geq 0}$ is 
characterized by the condition $\mathcal{D}=0$. 
In particular $M^{3}(-4,\beta)$ with \emph{positive} $\beta$ is isometric to 
the universal covering $\widetilde{\mathrm{SL}}_{2}\mathbb{R}$
of the special linear group equipped with the above metric, 
but 
\emph{not
isomorphic} to $\widetilde{\mathrm{SL}}_{2}\mathbb{R}$
as Lie groups. We here note that $\widetilde{\mathrm{SL}}_{2}\mathbb{R}$ is a 
\emph{unimodular} Lie group, while $\widetilde{G}(1,\beta)$ is
\emph{non-unimodular}.

We here mention that the 
non-unimodular Lie algebra $\mathfrak{g}(1,\beta)$ is  
isomorphic to the direct sum $\mathfrak{ga}(1)\oplus\mathbb{R}$. 
The Lie algebra $\mathfrak{ga}(1)$ of the orientation 
preserving affine transformation group $\mathrm{GA}^{+}(1)$
is given explicitly by
\[
\left\{\left.
\left(
\begin{array}{ccc}
v & u \\
0 & 0 
\end{array}
\right)
\>\right|
\>u,v\in\mathbb{R}
\right\}
\]
and the direct sum $\mathfrak{ga}(1)\oplus\mathbb{R}$ is spanned by the basis
\[
\hat{e}_1=\left(
\begin{array}{ccc}
0 & 1 & 0\\
0 & 0 & 0\\
0 & 0 & 0
\end{array}
\right),
\quad 
\hat{e}_2=\left(
\begin{array}{ccc}
1 & 0 & 0\\
0 & 0 & 0\\
0 & 0 & 0
\end{array}
\right),
\quad 
\hat{e}_3=\left(
\begin{array}{ccc}
0 & 0 & 0\\
0 & 0 & 0\\
0 & 0 & 1
\end{array}
\right).
\]
The commutation relation is 
\[
[\hat{e}_1,\hat{e}_2]=-\hat{e}_1.
\]
One can confirm that 
\[
\hat{e}_1\longmapsto 2e_2+2\beta e_3,
\quad 
\hat{e}_2\longmapsto \frac{1}{2}e_1,
\quad 
\hat{e}_3\longmapsto e_3 
\]
gives a Lie algebra isomorphism from 
$\mathfrak{ga}(1)\oplus\mathbb{R}$ to $\mathfrak{g}(1,\beta)$.
}
\end{example}

\medskip

Direct computation of $\nabla{R}$ on $G$
yields the following result.
\begin{proposition}\label{prop:4.7}
Let $(G,g)$ be a $3$-dimensional 
non-unimodular Lie group and use the notations introduced above.
Then $G$ is locally symmetric if and only if $\alpha=0$ or $(\alpha,\beta)=(1,0)$. 
\end{proposition}
One can confirm the following relations (\textit{cf}. \cite{I-UU07}):
\begin{align*}
& \rho_1=\rho_2\Longleftrightarrow \alpha=0,\>\>\mbox{or}\>\>1,
\\
& \rho_2=\rho_3\Longleftrightarrow 
\rho_1=\rho_3\Longleftrightarrow \alpha=0.
\end{align*}

\section{The homogeneous Riemannian structures on 
$3$-dimensional unimodular Lie groups}
\label{sec:5}

In this section we carry out the classification of 
homogeneous Riemannian structures on $3$-dimensional unimudular Lie groups 
whose structure constants $c_1$, $c_2$ and $c_3$ are 
\emph{distinct}. 

\subsection{Connection forms} 
Let $\{\theta^1,\theta^2,\theta^3\}$ be the 
orthonormal coframe field 
metrically dual to $\mathcal{E}=\{e_1,e_2,e_3\}$. 
Then the connection $1$-forms of $G$ relative to 
$\{\theta^1,\theta^2,\theta^3\}$ are 
\begin{equation}\label{eq:5.1}
\omega_{2}^{\>1}=-\mu_{3}\theta^3,
\quad 
\omega_{3}^{\>1}=\mu_{2}\theta^2,
\quad 
\omega_{3}^{\>2}=-\mu_{1}\theta^1.
\end{equation}
\subsection{The homogeneous Riemannian structures}
Let us prove the following theorem.

\begin{theorem}\label{thm:MainTheorem1}
The only homogeneous Riemannian 
structure a $3$-dimensional unimodular 
Lie group $G$ with distinct 
structure constants is $S=\nabla^{(-)}-\nabla$. 
The corresponding coset space representation is 
$G=G/\{\mathsf{e}\}$.
\end{theorem}
\begin{proof}
Let $S_{\flat}$ be a tensor field of 
type $(0,3)$ on a unimodular Lie group $G$ satisfying 
\eqref{eq:3.2} and express $S_{\flat}$ 
in the form \eqref{eq:3.3} by using 
$\{\theta^1,\theta^2,\theta^3\}$. 
Let $S$ be the tensor field 
of type $(1,2)$ associated with 
$S_{\flat}$ via \eqref{eq:3.4}. 
Assume that $S$ is a homogeneous Riemannian structure, then 
$\tilde{\nabla}$ satisfies 
$\tilde{\nabla}R=0$. Since 
$\{e_1,e_2,e_3\}$ diagonalizes the Ricci 
tensor field, we obtain the system 
\eqref{eq:3.6}--\eqref{eq:3.8}.

First we assume that all the principal Ricci curvatures are 
\emph{distinct}, then Lemma \ref{lem:3.1} together with 
\eqref{eq:5.1} yields
\begin{align*}
& S_{\flat}(X,e_1,e_2)=\omega_{2}^{\>1}(X)=-\mu_{3}\theta^{3}(X),
\\
& S_{\flat}(X,e_3,e_1)=-\omega_{3}^{\>1}(X)=-\mu_{2}\theta^{2}(X),
\\
& S_{\flat}(X,e_2,e_3)=\omega_{3}^{\>2}(X)=-\mu_{1}\theta^{1}(X).
\end{align*}
Comparing these with \eqref{eq:minusconnection}, we notice that $S_{\flat}$ coincides with $S_{\flat}^{(-)}$. 
Namely $\nabla+S$ is the Cartan-Schouten's $(-)$-connection.

Next we consider the case 
only two principal Ricci curvatures coincides.
Up to renumeration, we may assume that 
\[
\rho_1=\rho_2\not=\rho_3.
\]
Note that $G$ is pseudo-symmetric in the sense of Deszcz.  
Here we remark that 
\[
\rho_1=\rho_2\Longleftrightarrow \mu_3=0,\quad \mbox{or}\quad \mu_1=\mu_2.
\]
Moreover we notice that 
\[
\mu_1=\mu_2 \Longleftrightarrow c_1=c_2.
\]
Since we assumed that all of $c_1$, $c_2$ and $c_3$ are distinct, 
the only possibility for $\rho_1=\rho_2$ is 
$\mu_3=0$ which is equivalent to $c_3=c_1+c_2$. 
The homogeneous Riemannian structure $S$ has the form 
\[
S_{\flat}
=
2S_{\flat}(X,e_1,e_2)(\theta^1\wedge \theta^2)
-2\mu_{1}\theta^{1}\otimes(\theta^2\wedge \theta^3)
-2\mu_{2}\theta^{2}\otimes
(\theta^3\wedge \theta^1).
\]
To determine $S_{\flat}(X,e_1,e_2)$, we compute
$\tilde{\nabla}S$. The covariant form $S_{\flat}$ of $S$ is 
expressed as
\[
S_{\flat}
=2\sigma\otimes(\theta^1\wedge \theta^2)
-2\mu_{1}\theta^1\otimes(\theta^2\wedge \theta^3)
-2\mu_{2}\theta^2\otimes(\theta^3\wedge \theta^1),
\]
where the $1$-form $\sigma$ is defined by 
$\sigma(X)=S_{\flat}(X,e_1,e_2)$. 
The connection $1$-forms 
$\{\tilde{\omega}_{i}^{\>\>j}\}$ of $\tilde{\nabla}$ relative to 
$\{e_1,e_2,e_3\}$ are given by
\[
\tilde{\omega}_{i}^{\>\>j}(X)=g(\tilde{\nabla}_{X}e_i,e_j)
=\omega_{i}^{\>\>j}(X)+S(X,e_i,e_j).
\]
Hence 
\[
\tilde{\omega}_{1}^{\>\>2}(X)
=\sigma(X),
\quad
\tilde{\omega}_{1}^{\>\>3}(X)
=
\tilde{\omega}_{2}^{\>\>3}(X)=0,
\]
\[
\tilde{\nabla}_{X}\theta^1=
\sigma(X)\,\theta^2,
\quad 
\tilde{\nabla}_{X}\theta^2=-
\sigma(X)
\,\theta^1,
\quad 
\tilde{\nabla}_{X}\theta^3=0.
\]
By using these we obtain  
\[
\tilde{\nabla}_{X}S_{\flat}=
2(\tilde{\nabla}_{X}\sigma)(\theta^1\wedge \theta^2)
+2(\mu_{1}-\mu_{2})\sigma(X)
\{
\theta^1\otimes(\theta^1\wedge\theta^3)
-\theta^2\otimes(\theta^2\wedge\theta^3)
\}.
\]
Hence $\tilde{\nabla}S_{\flat}=0$ if and only if 
$\sigma=0$. 
Thus $S_{\flat}$ is 
given by
\[
S_{\flat}
=
-2\mu_{1}\theta^1\otimes(\theta^2\wedge \theta^3)
-2\mu_{2}\theta^2\otimes(\theta^3\wedge \theta^1).
\]
Hence 
$\tilde{\nabla}=\nabla+S_{\flat}$ coincides 
with $(-)$-connection. 
Since the curvature form $\tilde{\varOmega}$ 
of the $(-)$-connection vanishes, 
the coset space representation 
is $G/\{\mathsf{e}\}$.
\end{proof}
As pointed out by \cite{CFG},
$S$ is of type $\mathcal{T}_2$ 
if and only if $c_1+c_2+c_3=0$. 
For other cases, $S$ is of type 
$\mathcal{T}_2\oplus\mathcal{T}_3$. 
Non trivial homogeneous 
Riemannian structures of type $\mathcal{T}_2$ 
can exist only on $\mathrm{SL}_2\mathbb{R}$ and 
$\mathrm{SE}(1,1)$. For more detail, see
\cite{GGO2,I,KoTr}. 
%%%%%%%%%%%%%%%%%%%%%%%%%%%%%%%%%%%%%%%%%%%%
\section{The homogeneous Riemannian structures on 
$3$-dimensional non-unimodular Lie groups}\label{sec:6}
In this section we classify 
homogeneous Riemannian structures on $3$-dimensional non-unimudular Lie groups 
whose structure constants $\alpha$ and $\beta$ 
satisfy $\alpha\not=0$ and $(\alpha,\beta)\not=(1,0)$.

It should be remarked that homogeneous Riemannian structures 
of the hyperbolic $3$-space $\mathbb{H}^3(-1)$ ($\alpha=0$) and 
those of the Riemannian product 
$\mathbb{H}^2(-4)\times\mathbb{R}$ ($(\alpha,\beta)=(1,0)$) were classified in 
\cite{Abe} and \cite{Ohno}, respectively.

\subsection{Connection forms} 
Let $G$ be a $3$-dimensional unimodular 
Lie group whose Lie algebra is $\mathfrak{g}(\alpha,\beta)$. 
Let $\{\theta^1,\theta^2,\theta^3\}$ be the 
orthonormal coframe field 
metrically dual to $\mathcal{E}=\{e_1,e_2,e_3\}$. 
Then the connection $1$-forms of $G$ relative to 
$\{\theta^1,\theta^2,\theta^3\}$ are
\begin{equation}\label{eq:6.0}
\omega_{2}^{\>1}=(1+\alpha)\theta^2+\alpha\beta\theta^3,
\quad 
\omega_{3}^{\>1}=\alpha\beta\theta^2+(1-\alpha)\theta^3,
\quad 
\omega_{3}^{\>2}=-\beta\theta^1.
\end{equation}
\subsection{The homogeneous Riemannian structures}
As in the previous section, we take a 
tensor field $S_{\flat}$ of type $(0,3)$ 
satisfying \eqref{eq:3.2} and express it 
as in the form \eqref{eq:3.3}. 
By using the associated tensor field $S$
defined by \eqref{eq:3.4}, we give a 
linear connection $\tilde{\nabla}=\nabla+S$ 
on $G$. 
From the assumption that 
$G$ is \emph{not locally symmetric},
$\rho_1\neq\rho_3$ and $\rho_2\neq\rho_3$,
and the system for 
the parallelism $\tilde{\nabla}\mathrm{Ric}=0$ 
is deduced as
\begin{align}\label{eq:6.1}
& 
(\rho_1-\rho_2)S_{\flat}(X,e_1,e_2)=(\rho_1-\rho_2)\omega_{2}^{\>1}(X)
\\
& S_{\flat}(X,e_3,e_1)=-\omega_{3}^{\>1}(X)=
-\alpha\beta \theta^2(X)-(1-\alpha)\theta^3(X),
\\
& S_{\flat}(X,e_2,e_3)=\omega_{3}^{\>2}(X)=-\beta\theta^{1}(X).
\end{align}
Theorem 1.3 of \cite{CFG} has two cases (i) and (ii). 
Under the assumption $\nabla R\not=0$, 
the case (i) corresponds to $\alpha=1$ and $\beta\not=0$, equivalently 
$\rho_1=\rho_2$. 
The case (ii) corresponds to $\alpha\not=1$, equivalently 
$\rho_1\not=\rho_2$.

\subsection{The case $\rho_1\not=\rho_2$}
In this case, all of the 
principal Ricci curvatures are distinct. 
Thus from Lemma \ref{lem:3.1}, $S_{\flat}$ has the form
\[
S_{\flat}
=2\omega_{2}^{\,1}\otimes(\theta^1\wedge \theta^2)
+2\omega_{3}^{\>2}\otimes(\theta^2\wedge \theta^3)
+2\omega_{1}^{\>3}\otimes(\theta^3\wedge \theta^1)
\]
with $\alpha\not=1$.
This implies that $\tilde{\nabla}
=\nabla+S$ is the $(-)$-connection. 
Hence $S$ is a homogeneous Riemannian structure and has 
the form:
\[
S_{\flat}
=2\{(1+\alpha)\theta^2+\alpha\beta\theta^3\}
\otimes(\theta^1\wedge \theta^2)
-2\beta\theta^1\otimes(\theta^2\wedge \theta^3)
+2\{\alpha\beta \theta^2+
(1-\alpha)\theta^3\}\otimes(\theta^3\wedge \theta^1).
\]
Since the curvature form $\tilde{\varOmega}$ 
of the $(-)$-connection vanishes, 
the coset space representation 
is $G/\{\mathsf{e}\}$. 
Thus we retrive 
Theorem 1.3-(ii) of \cite{CFG} without the assumption of 
left invariance of $S$. As pointed out in \cite{CFG}, 
$S^{(-)}$ is of type $\mathcal{T}_1\oplus\mathcal{T}_2\oplus
\mathcal{T}_3$. In particular 
$S^{(-)}$ is of type $\mathcal{T}_2$ if and only if $\beta=0$. 
More precisely, 
on the semi-direct product group
$\tilde{G}(\alpha,0)=\mathbb{H}^2(-(1+\alpha)^2)
\times_{f_{\alpha}}\mathbb{R}$ exhibited 
in Example \ref{eg:4.2}, 
the homogeneous Riemannian structure 
associated with $(-)$-connection is of type 
$\mathcal{T}_2$.

\subsection{The case $\rho_1=\rho_2$}
In this case the simply connected 
Lie group $\widetilde{G}=
\widetilde{G}(1,\beta)$ corresponds to 
$\mathfrak{g}=\mathfrak{g}(1,\beta)$ with $\beta>0$
is isometric to 
the universal covering $\widetilde{\mathrm{SL}}_2\mathbb{R}$ 
of $\mathrm{SL}_2\mathbb{R}$ (see Example \ref{eg:4.3}). 
It should be emphasized that 
the homogeneous Riemannian space 
$\widetilde{G}(1,\beta)$ has $4$-dimensional 
isometry group. Indeed, $\widetilde{G}(1,\beta)$ is 
represented by $\widetilde{\mathrm{SL}}_2\mathbb{R}\times
\mathrm{SO}(2)/\mathrm{SO}(2)$ (see \cite{IO}).

The covariant form $S_{\flat}$ of a 
homogeneous Riemannian structure $S$ 
is expressed as
\begin{equation}\label{eq:Sflat}
  S_{\flat}
=2\sigma\otimes(\theta^1\wedge \theta^2)
-2\beta\theta^1\otimes(\theta^2\wedge \theta^3)
-2\beta\theta^2\otimes(\theta^3\wedge \theta^1),
\end{equation}
where the $1$-form $\sigma$ is defined by
\[
\sigma(X)=S_{\flat}(X,e_1,e_2).
\]
As in the previous section, 
the connections forms of $\tilde{\nabla}
=\nabla+S$ are computed as
\[
\tilde{\omega}_{1}^{\>2}(X)=\omega_{1}^{\>2}(X)+\sigma(X)=
-2\theta^{2}(X)-\beta\theta^{3}(X)+\sigma(X),
\quad
\tilde{\omega}_{1}^{\>3}(X)=
\tilde{\omega}_{2}^{\>3}(X)=0.
\]
Hence 
\begin{align*}
\tilde{\nabla}_{X}\theta^1
=&\left(
-2\theta^{2}(X)-\beta\theta^{3}(X)+\sigma(X)\right)
\theta^2,
\\
\tilde{\nabla}_{X}\theta^2
=&-\left(
-2\theta^{2}(X)-\beta\theta^{3}(X)+\sigma(X)\right)
\theta^1,
\\
\tilde{\nabla}_{X}\theta^3=&0.
\end{align*}
From these we have 
\[
(\tilde{\nabla}_{X}S_{\flat})
=2(\tilde{\nabla}_{X}\sigma)\otimes
(\theta^1\wedge\theta^2).
\]
Hence $\tilde{\nabla}S=0$ if and 
only if $\tilde{\nabla}\sigma=0$. 
Represent $\sigma$ as
\[
\sigma=\sigma_1\theta^1+\sigma_2\theta^2+
\sigma_3\theta^3
\]
and put 
\[
\tilde{\sigma}=-2\theta^2-\beta\theta^3+\sigma=\sigma_{1}\theta^1
+(\sigma_2-2)\theta^2+(\sigma^3-\beta)\theta^3,
\]
then
\begin{align*}
\tilde{\nabla}_{X}\sigma=&
\left\{
(\mathrm{d}\sigma_1)(X)-\tilde{\sigma}(X)\sigma_2
\right\}\theta^1
+
\left\{
(\mathrm{d}\sigma_2)(X)+\tilde{\sigma}(X)\sigma_1
\right\}\theta^2
+
(\mathrm{d}\sigma_3)(X)
\,\theta^3.
\end{align*}
Hence we obtain the differential system:
\begin{equation}\label{eq:theSystem}
\mathrm{d}\sigma_1=\sigma_{2}\tilde{\sigma},
\quad 
\mathrm{d}\sigma_2=-\sigma_{1}\tilde{\sigma},
\quad 
\mathrm{d}\sigma_3=0.
\end{equation}
The integrability condition of 
the system \eqref{eq:theSystem} is
\[
\mathrm{d}(\mathrm{d}\sigma_1)=
\mathrm{d}(\mathrm{d}\sigma_2)=0.
\]
Thus we obtain
\[
\sigma_1\,\mathrm{d}\tilde{\sigma}=0, 
\quad 
\sigma_2\,\mathrm{d}\tilde{\sigma}=0.
\]
\subsubsection{$\mathrm{SO}(2)$-isotropy}
First we observe the trivial solution 
$\sigma_1=\sigma_2=0$ (identically zero). 
In this case 
$\sigma=\sigma_3\theta^3$ and $\sigma_3$ is a 
constant. Set $\sigma_3=-r$ ($r\in\mathbb{R}$), then 
we obtain a 
one-parameter family 
\begin{equation}\label{eq:SO(2)}
S_{\flat}
=-2r\theta^3\otimes
(\theta^1\wedge\theta^2)
-2\beta\theta^1\otimes(\theta^2\wedge\theta^3)
-2\beta\theta^2\otimes(\theta^3\wedge\theta^1)
\end{equation}
of homogeneous Riemannian structures. 
This family coincides with 
the one 
of \cite[Theorem 1.3-(i)]{CFG}. 
The homogeneous Riemannian structure $S$ is 
of type $\mathcal{T}_2\oplus\mathcal{T}_3$. 
In particular $S$ is of type $\mathcal{T}_2$ if and 
only if $r=-2\beta$. Moreover $S$ is of type 
$\mathcal{T}_3$ if and only if $r=\beta$. 
The connection form $\tilde{\omega}$ and 
the curvature form $\tilde{\varOmega}$ of 
$\tilde{\nabla}=\nabla+S$ are given by
\[
\tilde{\omega}
=\left(
\begin{array}{ccc}
0 & 2\theta^2+(r+\beta)\theta^3 & 0\\
-2\theta^2-(r+\beta)\theta^3 & 0 & 0\\
0 & 0 & 0
\end{array}
\right),
\]
\[
\tilde{\varOmega}
=-2(\beta(r+\beta)+2)\left(
\begin{array}{ccc}
0 & \theta^1\wedge\theta^2 & 0\\
-\theta^1\wedge\theta^2 & 0 & 0\\
0 & 0 & 0
\end{array}
\right).
\]
Hence the holonomy algebra 
$\mathfrak{h}$ is isomorphic to $\mathfrak{so}(2)$ unless 
$r=-(\beta^2+2)/\beta$. 
Since 
\[
\tilde{\varOmega}_{1}^{\>2}=\tilde{R}_{1212}\,\theta^1\wedge\theta^2
=2(\beta(r+\beta)+2)\theta^1\wedge\theta^2,
\quad 
\tilde{\varOmega}_{1}^{\>3}=
\tilde{\varOmega}_{2}^{\>3}=0,
\]
The curvature operator $\tilde{R}(e_1,e_2)$ acts on 
$\mathfrak{g}$ as
\[
\tilde{R}(e_1,e_2)e_1=2(\beta(r+\beta)+2)e_2,
\quad 
\tilde{R}(e_1,e_2)e_2=-2(\beta(r+\beta)+2)e_1,
\quad 
\tilde{R}(e_1,e_2)e_3=0
\]
Thus $\tilde{R}(e_1,e_2)$ is identified with the matrix
\[
2(\beta(r+\beta)+2)\left(
\begin{array}{rrr}
0 &-1 & 0\\
1 & 0 & 0\\
0 & 0 & 0
\end{array}
\right)\in\mathfrak{so}(2)\subset\mathfrak{so}(3).
\]
The Lie algebra $\mathfrak{l}=\mathfrak{h}+\mathfrak{g}$ is 
generated by the basis 
$\{e_1,e_2,e_3,e_4\}$ with 
\[
e_4
=-\tilde{R}(e_1,e_2)=2(\beta(r+\beta)+2)\left(
\begin{array}{rrr}
0 & 1 & 0\\
-1 & 0 & 0\\
0 & 0 & 0
\end{array}
\right)\in\mathfrak{h}=\mathfrak{so}(2).
\]
Since 
\[
S(e_1)e_2=-\beta e_3,\quad S(e_2)e_1=\beta e_3,
\quad 
S(e_2)e_3=-\beta e_1,\quad S(e_3)e_2=re_1,
\]
\[
S(e_3)e_1=-re_1,\quad 
S(e_1)e_3=\beta e_2,
\]
from \eqref{eq:m-bla},
\begin{align*}
&[e_1,e_2]_{\mathfrak l}=-\tilde{R}(e_1,e_2)-S(e_1)e_2+S(e_2)e_1
=2\beta e_3+e_4.
\\
& [e_2,e_3]_{\mathfrak l}=-\tilde{R}(e_2,e_3)-S(e_2)e_3+S(e_3)e_2
=(r+\beta)e_1,
\\
& 
[e_3,e_1]_{\mathfrak l}=-\tilde{R}(e_3,e_1)-S(e_3)e_1+S(e_1)e_3
=(r+\beta)e_2.
\end{align*}
Next, by \eqref{eq:hm-bla} 
$\mathrm{ad}(e_4)$ is described as 
(\textit{cf.} \cite[pp.~77-79]{TV}):
\[
[e_4,e_1]_{\mathfrak l}=-2(\beta(r+\beta)+2)e_2,\quad 
[e_4,e_2]_{\mathfrak l}=2(\beta(r+\beta)+2)e_1,\quad 
[e_4,e_3]_{\mathfrak l}=0.
\]
When $r=-(\beta^2+2)/\beta$, $\mathfrak{h}=\{0\}$ and 
$\mathfrak{l}$ is $\mathfrak{sl}_2\mathbb{R}$. In case 
$r\not=-(\beta^2+2)/\beta$, $e_4\not=0$ and we set
\[
\tilde{e}_1=e_1,
\quad 
\tilde{e}_2=e_2,
\quad 
\tilde{e}_3=2\beta e_3+e_4,
\]
then we get
\[
[\tilde{e}_1,\tilde{e}_2]_{\mathfrak l}=
\tilde{e}_3,
\quad 
[\tilde{e}_2,\tilde{e}_3]_{\mathfrak l}=-4\tilde{e}_1,
\quad 
[\tilde{e}_3,\tilde{e}_1]_{\mathfrak l}=-4\tilde{e}_2.
\]
Thus the Lie subalgebra generated by $\{\tilde{e}_1,\tilde{e}_2,\tilde{e}_3\}$ 
is isomorphic to $\mathfrak{sl}_2\mathbb{R}$ and hence 
$\mathfrak{l}$ is the semi-direct sum 
$\mathfrak{so}(2)+\mathfrak{sl}_2\mathbb{R}$. 

If we set
\[
\tilde{e}_4=e_4-\frac{1}{2}(\beta(r+\beta)+2)e_3,
\]
then 
\[
[\tilde{e}_4,\tilde{e}_1]_{\mathfrak{l}}=
[\tilde{e}_4,\tilde{e}_2]_{\mathfrak{l}}=
[\tilde{e}_4,\tilde{e}_3]_{\mathfrak{l}}=0.
\]
Thus $\mathfrak{l}$ is isomorphic to 
the direct sum $\mathfrak{sl}_2\mathbb{R}
\oplus\mathbb{R}\tilde{e}_4$.

\subsubsection{Trivial isotropy}
Next we assume that not of both  
$\sigma_1$ and $\sigma_2$ are 
identically zero. Then we get 
\[
\mathrm{d}\tilde{\sigma}
=\mathrm{d}(-2\theta^2-\beta\theta^3+\sigma)=0.
\]
Hence there exists a smooth function $f$ such that 
\[
-2\theta^2-\beta\theta^3+\sigma=\mathrm{d}f.
\]
From this equation we have
\[
e_{1}(f)=\sigma_1,
\quad 
e_{2}(f)=\sigma_2-2,
\quad 
e_{3}(f)=\sigma_3-\beta.
\]
By using these, 
\[
[e_1,e_2]f=e_{1}(e_{2}(f))
-e_{2}(e_{1}(f))=
(\mathrm{d}\sigma_2)(e_1)
-
(\mathrm{d}\sigma_1)(e_2)
=-\sigma_{1}^2-\sigma_2^2+2\sigma_2.
\]
On the other hand, from 
$[e_1,e_2]=2e_2+2\beta e_3$,
we get
\[
[e_1,e_2]f=2(\sigma_2-2)+2\beta(\sigma_3-\beta).
\]
Combining these two equations, we obtain
\[
\sigma_1^2+\sigma_2^2=4-2\beta(\sigma_3-\beta).
\]
Next, 
\[
[e_2,e_3]f=e_2(e_3f)-e_3(e_2f)=-\sigma_1(\sigma_3-\beta),
\quad 
[e_3,e_1]f=\sigma_2(\sigma_3-\beta).
\]
Since $[e_2,e_3]=[e_3,e_1]=0$, 
we get $\sigma_3=\beta$. Hence 
$\sigma_1^2+\sigma_2^2=4$. 
Thus, there exists a smooth function $\phi$ such that 
\begin{equation}\label{eq:phi}
\sigma_1=2\cos\phi,
\quad 
\sigma_2=2\sin\phi
\end{equation}
and hence 
\[
\sigma=2(\cos\phi \theta^1+\sin\phi \theta^2)+\beta\theta^3.
\]
From 
the differential system \eqref{eq:theSystem}, the function $\phi$ satisfies the differential equation:
\begin{equation}\label{eq:ODEphi}
\mathrm{d}\phi
=-2\cos\phi\, \theta^1+2(1-\sin\phi)\,\theta^2.
\end{equation}
Conversely, take a solution $\phi$ to 
\eqref{eq:ODEphi} and define $\sigma_1$ and $\sigma_2$ by 
\eqref{eq:phi}. Then $\sigma=\sigma_1\theta^1+\sigma_2\theta^2
+\beta\theta^3$ is a solution to the system 
\eqref{eq:theSystem}, where $\tilde{\sigma}=\sigma-2\theta^2-\beta\theta^3$.

One can see that the only constant solution 
$\phi$ to \eqref{eq:ODEphi} is $\phi=\pi/2$. 
In this case 
$\sigma=2\theta^2+\beta\theta^{3}=\omega_{2}^{\>1}$. 
Hence 
\[
S_{\flat}=2\omega_{2}^{\>1}\otimes
(\theta^1\wedge\theta^2)
-2\beta\theta^1\otimes(\theta^2\wedge\theta^3)
-2\beta\theta^2\otimes(\theta^3\wedge\theta^1).
\]
Thus $\nabla+S$ is nothing but the 
$(-)$-connection.

\smallskip

Let us continue to determine the homogeneous 
Riemannian structure $S$ with \emph{non-constant} $\phi$. 
From the first structure equations we get
\[
\mathrm{d}\theta^1=0,\quad 
\mathrm{d}\theta^2=-2(\theta^1\wedge\theta^2),
\quad 
\mathrm{d}\theta^3=-2\beta(\theta^1\wedge\theta^2).
\]
By using these we deduce that 
\[
\mathrm{d}\tilde{\sigma}=0.
\]
Since $\tilde{\sigma}=\tilde{\omega}_{1}^{\>2}$ 
and $\tilde{\omega}_{3}^{\>1}=\tilde{\omega}_{3}^{\>2}=0$, 
the curvature form $\widetilde{\varOmega}$ of 
the Ambrose-Singer connection 
$\tilde{\nabla}=\nabla+S$ is 
given by
\[
\tilde{\varOmega}
=\mathrm{d}\tilde{\omega}+\tilde{\omega}\wedge\tilde{\omega}
=-\mathrm{d}\tilde{\sigma}(\theta^1\wedge \theta^2)=0.
\]
Hence the curvature $\tilde{R}$ of $\tilde{\nabla}$ vanishes. 
From the equation \eqref{eq:Sflat}, the components of $S$ are computed as 
\[
\begin{array}{lll}
S(e_1)e_1=\sigma_1 e_2, 
&
S(e_1)e_2=-\sigma_{1}e_1-\beta e_3,
&
S(e_1)e_3=\beta e_2,\\
S(e_2)e_1=\sigma_2 e_2+\beta e_3, 
&
S(e_2)e_2=-\sigma_{2}e_1,
&
S(e_2)e_3=-\beta e_1,\\
S(e_3)e_1=\beta e_2, 
&
S(e_3)e_2=-\beta e_1,
&
S(e_3)e_3=0.
\end{array}
\]
Hence $\tilde{\nabla}$ is 
described as
\begin{equation}\label{eq:nonuniAS}
\begin{array}{lll}
\tilde{\nabla}_{e_1}e_1=\sigma_1 e_2, 
&
\tilde{\nabla}_{e_1}e_2=-\sigma_{1}e_1,
&
\tilde{\nabla}_{e_1}e_3=0,\\
\tilde{\nabla}_{e_2}e_1=(\sigma_2-2) e_2, 
&
\tilde{\nabla}_{e_2}e_2=-(\sigma_{2}-2)e_1,
&
\tilde{\nabla}_{e_2}e_3=0,\\
\tilde{\nabla}_{e_3}e_1=0, 
&
\tilde{\nabla}_{e_3}e_2=0,
&
\tilde{\nabla}_{e_3}e_3=0.
\end{array}
\end{equation}
Here we note that
\[
0\leq \sigma_1^2+(\sigma_2-2)^2=
\sigma_1^2+\sigma_2^2-4\sigma_2+4
=4(2-\sigma_2).
\]
Hence $0\leq \sigma_2\leq 2$. 
In particular, when $\sigma_2=2$, then 
$\sigma_1=0$ (thus $\phi=\pi/2$) and hence 
$\tilde{\nabla}=\nabla^{(-)}$ as we mentioned before. 

Let us investigate the case $\sigma_2-2$ is not identically zero, 
equivalently $\phi$ is non-constant. 
Change the 
left invariant 
orthonormal frame field $\{e_1,e_2,e_3\}$ to 
the following orthonormal frame field:
\[
\tilde{e}_1=
\frac{1}{2}
(\sigma_2e_1-\sigma_1e_2),
\quad 
\tilde{e}_2=
\frac{1}{2}
(\sigma_1e_1+\sigma_2e_2),
\quad 
\tilde{e}_3=e_3.
\]
From the equations \eqref{eq:phi}, \eqref{eq:ODEphi} and \eqref{eq:nonuniAS}, 
we can check that 
\[
\tilde{\nabla}_{\tilde{e}_i}\tilde{e}_j=0,
\quad i,j=1,2,3.
\]
It should be remarked that $\tilde{e}_1$ and 
$\tilde{e}_2$ are \emph{not} left invariant.

Let us determine 
the Lie group $L$ acting 
isometrically and transitively 
on $G$ corresponding 
to $S$ and its 
reductive decomposition 
$\mathfrak{l}=\tilde{\mathfrak h}\oplus\mathfrak{g}$. 
Here $\mathfrak{g}$ is regarded as the 
tangent space of $G$ at the identity $\mathsf{e}$. 
Since the isotropy algebra $\tilde{\mathfrak h}$ is 
the holonomy algebra 
spanned by $\tilde{R}$, we have $\tilde{\mathfrak h}
=\{0\}$. Hence $\mathfrak{l}=\mathfrak{g}$ as an inner product space. 
The Lie bracket $[\cdot,\cdot]_{\mathfrak l}$ of $L$ is determined by 
\eqref{eq:m-bla} as follows:

\begin{align*}
& 
[e_1,e_2]_{\mathfrak l}=(-S(e_1)e_2+S(e_2)e_1)\vert_{\mathsf e}
=\sigma_1(\mathsf{e})e_1+\sigma_2(\mathsf{e})e_2+2\beta e_3,
\\
& 
[e_2,e_3]_{\mathfrak l}=(-S(e_2)e_3+S(e_3)e_2)\vert_{\mathsf e}
=0,
\\
& 
[e_3,e_1]_{\mathfrak l}=(-S(e_3)e_1+S(e_1)e_3)\vert_{\mathsf e}
=0.
\end{align*}
With respect to the new basis 
$\{\tilde{e}_1,\tilde{e}_2,\tilde{e}_3\}$, 
the commutation relations of $\mathfrak{l}$ are described as
we obtain
\begin{align*}
& [\tilde{e}_1,\tilde{e}_2]_{\mathfrak l}=
\frac{1}{4}[
\sigma_2(\mathsf{e})e_1-\sigma_1(\mathsf{e})e_2
,\sigma_1(\mathsf{e})e_1+\sigma_2(\mathsf{e})e_2]
=
2\tilde{e}_2+2\beta\tilde{e}_3,
\\
& [\tilde{e}_2,\tilde{e}_3]_{\mathfrak l}=
\frac{1}{2}[\sigma_1(\mathsf{e})e_1+\sigma_2(\mathsf{e})e_2,e_3]_{\mathfrak{l}}
=0,
\\
& [\tilde{e}_3,\tilde{e}_1]_{\mathfrak l}=
\frac{1}{2}[e_3,\sigma_2(\mathsf{e})e_1-\sigma_1(\mathsf{e})e_2]_{\mathfrak{l}}
=0.
\end{align*}
Since $\tilde{\nabla}_{\tilde{e}_i}\tilde{e}_j=0$, 
the connection $\tilde{\nabla}=\nabla+S$ is 
the $(-)$-connection of 
the non-unimodular Lie group 
corresponding to $\mathfrak{l}$. 
Thus $(\mathfrak{l},[\cdot,\cdot]_{\mathfrak l})$ 
is isomorphic to 
$\mathfrak{g}(1,\beta)$. This fact means that 
$(\widetilde{G}(1,\beta),g,S)$ is isomorphic to 
$(\widetilde{G}(1,\beta),g,S^{(-)})$ from Theorem \ref{thm:hom}.

Thus we arrive at the following theorem which 
improves \cite[Theorem 1.3]{CFG}.
\smallskip
\begin{theorem}\label{thm:MainTheorem2}
Up to isomorphisms, the only homogeneous Riemannian 
structure on a non-locally symmetric $3$-dimensional non-unimodular 
Lie group $G$ with Lie algebra $\mathfrak{g}(1,\beta)$ $(\beta\not=0)$ is one of the following homogeneous structures{\rm:}
\begin{enumerate} 
\item a member of the following one-parameter family{\rm:}
\[
S^{(r)}=-2r\theta^3\otimes(\theta^1\wedge\theta^2)
-2\beta\theta^1\otimes(\theta^2\wedge\theta^3)
-2\beta\theta^2\otimes(\theta^3\wedge\theta^1),
\quad r\in\mathbb{R}.
\]
The coset space representation is
\[
G=\begin{cases}
L\times\mathrm{SO}(2)/\mathrm{SO}(2), & r\not=-(\beta^2+2)/\beta,
\\
L/\{\mathsf{e}\}, & r=-(\beta^2+2)/\beta.
\end{cases}
\]
Here $L$ is a Lie groups which acts transitively and isometrically on $G$ and 
its Lie algebra is isomorphic to $\mathfrak{sl}_2\mathbb{R}$ as a Lie algebra. 
In particular, if $G$ is simply connected, then $G$ is represented as
\[
G=\begin{cases}
\widetilde{\mathrm{SL}}_2\mathbb{R}\times\mathrm{SO}(2)/\mathrm{SO}(2), & r\not=-(\beta^2+2)/\beta,
\\
\widetilde{\mathrm{SL}}_2\mathbb{R}/\{\mathsf{e}\}, & r=-(\beta^2+2)/\beta.
\end{cases}
\]
\item
$S=\nabla^{(-)}-\nabla$. 
The corresponding coset space representation is 
$G=G/\{\mathsf{e}\}$.
\end{enumerate}
\end{theorem}

\section{Application to contact geometry and CR-geometry}
The special unitary group $\mathrm{SU}(2)$ equipped with 
naturally reductive left invariant metric 
admits compatible almost contact structure. 
Gadea and Oubi{\~n}a \cite{GO} proved that 
the homogeneous Riemannian structures 
on naturally reductive $\mathrm{SU}(2)$ 
are homogeneous almost contact structures. 
In this section we mention 
relations of our results to contact geometry. 
\subsection{Contact $3$-manifolds}
Let $M$ be a $3$-manifold, 
then a triplet $(\varphi,\xi,\eta)$ 
consisting of an endomorphism field $\varphi$, 
a vector field $\xi$ and a $1$-form $\eta$ is said to be 
an \emph{almost contact structure} 
if it satisfies
\[
\varphi^2=-\mathrm{I}+\eta\otimes\xi,\quad \eta(\xi)=1.
\]
Then one can deduce that 
\[
\eta\circ \varphi=0,
\quad \varphi\xi=0.
\]
Now let $(M,g)$ be a Riemannian 
$3$-manifold, then an almost contact structure 
$(\varphi,\xi,\eta)$ is said to be 
\emph{compatible} to $g$ if it satisfies
\[
g(\varphi X,\varphi Y)=g(X,Y)-\eta(X)\eta(Y)
\]
for all vector fields $X$ and $Y$ on $M$. 
The resulting $3$-manifold $(M;\varphi,\xi,\eta,g)$ 
together with structure tensor fields $(\varphi,\xi,\eta,g)$ 
is called an \emph{almost contact Riemannian $3$-manifold}. 
It is known that every oriented Riemannian $3$-manifold 
$(M,g)$ admits an almost contact structure 
compatible to the metric $g$ and orientation.

Let $M$ be a $3$-manifold. A 
\emph{contact form} is a $1$-form $\eta$ on $M$ satisfying 
$\mathrm{d}\eta\wedge\eta\not=0$ on $M$. 
A $3$-manifold $(M,\eta)$ together with a contact form 
is called a \emph{contact $3$-manifold}. 
On a contact $3$-manifold $(M,\eta)$, there exists a 
unique vector field satisfying 
$\eta(\xi)=1$ and $\mathrm{d}\eta(\xi,\cdot)=0$. 
The vector field $\xi$ is called the 
\emph{Reeb vector field} of $(M,\eta)$. 
Moreover there exists an endomorphism field $\varphi$ and a 
Riemannian metric such that 
$(M,\varphi,\xi,\eta,g)$ is an almost contact 
Riemannian $3$-manifold and additionally 
satisfying the \emph{contact metric condition}:
\[
\mathrm{d}\eta(X,Y)=g(X,\varphi Y)
\]
for all vector fields $X$ and $Y$ on $M$. 
The resulting $3$-manifold $(M;\varphi,\xi,\eta,g)$ 
together with structure tensor fields $(\varphi,\xi,\eta,g)$ 
is called a \emph{contact Riemannian $3$-manifold}.
It is known that every contact Riemannian $3$-manifold is a strongly pseudo-convex CR-manifold.

We may relax the contact metric condition as
\begin{equation}\label{eq:contactmetricweak}
\mathrm{d}\eta(X,Y)=\beta g(X,\varphi Y)
\end{equation}
for positive constant $\beta$ (see \cite{MPS}).

A contact Riemannian $3$-manifold 
is said to be a \emph{Sasakian $3$-manifold} if it satisfies 
\[
(\nabla_{X}\varphi)Y=g(X,Y)\xi-\eta(Y)X
\]
for all vector fields $X$ and $Y$ on $M$. 
More generally, $M$ is said to be 
a $\beta$-Sasakian $3$-manifold if it satisfies
\[
(\nabla_{X}\varphi)Y=\beta\{g(X,Y)\xi-\eta(Y)X\}
\]
for all vector fields $X$ and $Y$ on $M$. Here 
$\beta$ is a positive constant. 
Note that if $M$ is $\beta$-Sasakian, then 
it satisfies \eqref{eq:contactmetricweak}. 
In particular, Sasakian $3$-manifolds are 
contact Riemannian. 

On a Sasakian $3$-manifold $M$, 
a plane section is said to be 
\emph{holomorphic} if it is orthogonal to $\xi$. 
The sectional curvature function 
of holomorphic planes is 
regarded as a smooth function on $M$ and 
called the \emph{holomorphic sectional curvature} 
(or $\varphi$-sectional curvature). A Sasakian 
$3$-manifold is said to be a \emph{Sasakian space form} 
if it is complete and of constant holomorphic sectional curvature.

On a contact Riemannian $3$-manifold $M$, 
an endomorphism field $h$ defined by 
$h=\pounds_{\xi}\varphi/2$ is self-adjoint with respect to $g$. Here 
$\pounds_{\xi}$ denotes the Lie differentiation by $\xi$. 
One can introduce a linear connection 
$\hat{\nabla}$ by
\begin{equation}
\hat{\nabla}_{X}Y=
\nabla_{X}Y-g(\varphi(\mathrm{I}+h)X,Y)\xi+\eta(X)\varphi{Y}+\eta(Y)\varphi(\mathrm{I}+h)X.
\end{equation}
The linear connection $\hat{\nabla}$ is called the
\emph{Tanaka-Webster connection}.

\subsection{Contact $(\kappa,\mu)$-spaces}

According to \cite{BKP}, 
a contact Riemannian $3$-manifold $M$ is said to be a 
\emph{contact} $(\kappa,\mu)$-\emph{space} if there 
exist constants $\kappa$ and $\mu$ such that 
\[
R(X,Y)\xi=(\kappa\mathrm{I}+\mu h)(\eta(Y)X-\eta(X)Y)
\]
holds for all vector fields $X$ and $Y$ on $M$.
It should be remarked that $\kappa \leq 1$ and 
if $\kappa=1$, then $h=0$ and hence $M$ is Sasakian. 

In CR-geometry, $3$-dimensional Sasakian space forms 
together with non-Sasakian contact $(\kappa,\mu)$-spaces 
are characterized as $3$-dimensional CR-symmetric spaces 
in the sense of \cite{KaZa} (see also \cite{DL}).

Blair, Koufogiorgos and Papantoniou \cite{BKP} 
proved the local homogeneity of  
non-Sasakian $3$-dimensional contact $(\kappa,\mu)$-spaces.
Moreover a non-Sasakian $3$-dimensional contact $(\kappa,\mu)$-space 
$M$ is locally isometric to $3$-dimensional 
unimodular Lie group $G$.

Let $G$ be a $3$-dimensional unimodular 
Lie group. We introduce a left invariant almost contact structure $(\varphi,
\xi,\eta)$ on $G$ by
\[
\xi=e_1,\quad \eta=\theta^1,
\quad 
\varphi{e}_2=e_{3},
\quad
\varphi{e}_{3}=-e_2,\quad
\varphi \xi=0. 
\]
One can see that 
$(\varphi,
\xi,\eta)$ is compatible to 
$g$. Moreover if we choose $c_1=2$ 
(see Table \ref{table2}), then 
$(G,\varphi,\xi,\eta,g)$ is a homogeneous 
contact Riemannian $3$-manifold. 
The endomorphism field $h$ is given by
\[
he_2=-\frac{1}{2}(c_2-c_3)e_2,
\quad 
he_3=\frac{1}{2}(c_2-c_3)e_3.
\]
In particular $G$ is Sasakian if and only if $c_2=c_3$. 
On the other hand $G$ is flat if and only if 
$(c_2,c_3)=(2,0)$ or $(c_2,c_3)=(0,2)$.

\begin{proposition}
If a unimodular Lie group $G$ is non-Sasakian, then
$G$ is a contact $(\kappa,\mu)$-space with
\[
\kappa=1-\frac{1}{4}(c_{2}-c_{3})^{2},\quad 
\mu=2-(c_2+c_3).
\]
The possible Lie algebras are 
$\mathfrak{su}(2)$,
$\mathfrak{sl}_2\mathbb{R}$,
$\mathfrak{se}(2)$ and 
$\mathfrak{se}(1,1)$. 

Conversely every $3$-dimensional non-Sasakian 
contact $(\kappa,\mu)$-space is locally 
isomorphic to those unimodular Lie groups as 
contact Riemannian $3$-manifolds.
\end{proposition}
Boeckx \cite{Boeckx} proved the local homogeneity 
of contact $(\kappa,\mu)$-spaces of arbitrary odd dimension. 
The key clue of Boeckx's proof is 
the fact that
\[
S^{\mathrm B}(X)Y=-g(\varphi(\mathrm{I}+h)X,Y)\xi
+\frac{\mu}{2}\eta(X)\varphi{Y}+\eta(Y)\varphi(\mathrm{I}+h)X
\]
is a homogeneous Riemannian structure on arbitrary 
non-Sasakian contact $(\kappa,\mu)$-spaces.
Comparing this formula with Theorem \ref{thm:MainTheorem1}, we obtain the 
following result.
\smallskip

\begin{corollary}
The only homogeneous Riemannian structure 
of a $3$-dimensional
non-flat and non-Sasakian contact $(\kappa,\mu)$-space 
is the one derived from the $(-)$-connection.
\end{corollary}
\smallskip

Indeed Boeckx's connection 
$\nabla+S^{\mathrm B}$ coincides with $(-)$-connection. 
Note that Boeckx's connection coincides with Tanaka-Webster 
connection if and only if 
$\mu=2$, 
\textit{i.e.}, $c_2+c_3=0$. Thus the possible 
Lie algebras are $\mathfrak{sl}_2\mathbb{R}$ or 
the Heisenberg algebra. 
In particular when 
$(c_1,c_2,c_3)=(2,1,-1)$, then 
the Lie algebra is $\mathfrak{sl}_2\mathbb{R}$ and 
the corresponding Lie group $G$ admits 
totally geodesic surfaces (see Table \ref{table3}). 
\medskip

\begin{corollary} 
For a contact Riemannian structures on 
$\widetilde{\mathrm{SL}}_2\mathbb{R}$ with 
structure constants $(2,c_2,c_3)$ such that 
$c_2\not=c_3$, the following 
properties are mutually equivalent{\rm:}
\begin{itemize}
\item $\widetilde{\mathrm{SL}}_2\mathbb{R}$ admits 
totally geodesic surfaces.
\item it is a $(0,2)$-space.
\item The $(-)$-connection coincides with Tanaka-Webster connection. 
\end{itemize}
\end{corollary}
\smallskip

On the other hand, $S^{(-)}$ is of type $\mathcal{T}_2$ 
if and only if $\mu=4$.

\smallskip

Note that on $\mathrm{SE}(2)$ equipped with 
a left invariant flat contact Riemannian structure, 
$S^{\mathrm B}$ coincides with $(-)$-connection.

\smallskip

\subsection{Sasakian $3$-manifolds}
On a non-unimodular Lie group $G$ with 
Lie algebra $\mathfrak{g}(1,\beta)$ 
with $\beta>0$ exhibited in 
Example \ref{eg:4.3}, we can introduce a left 
invariant 
almost contact structure $(\varphi,\xi,\eta)$ by
\[
\xi=e_3,\quad \eta=\theta^3,
\quad 
\varphi{e}_1=e_{2},
\quad
\varphi{e}_{2}=-e_1,\quad
\varphi \xi=0. 
\]
Then $(G,\varphi,\xi,\eta,g)$ is a 
$\beta$-Sasakian manifold. 
The homogeneous Riemannian structure 
\eqref{eq:SO(2)} is 
expressed as
\begin{equation}\label{eq:Okumura}
S^{(r)}(X)Y=\beta g(X,Y)\xi-r\eta(X)\varphi Y+\beta\eta(Y)\varphi X.
\end{equation}
In particular, when $\beta=1$, $G(0,1)$ is a Sasakian $3$-manifold 
of constant holomorphic sectional curvature $-7$. 
The one-parameter family 
$\{\nabla+S^{(r)}\}_{r\in\mathbb{R}}$ 
of Ambrose-Singer connections coincides with 
the one-parameter family of linear connections 
of Okumura's \cite{Okumura}. 
In particular, $\nabla+S^{(-1)}$ is 
the Tanaka-Webster connection
(see also \cite{I-UU07,IO}).

\end{document}